\documentclass[11pt,reqno,a4paper]{amsart}
\usepackage{amssymb}
\usepackage{fancyhdr}
\usepackage{bbold}
\usepackage{ marvosym }\usepackage{stmaryrd}

\usepackage[
    backend=biber,
    style=numeric,
    natbib=false,
    url=false, 
    doi=false,
    eprint=true,
    maxnames=50
]{biblatex}

\usepackage{mathrsfs}
\usepackage{ esint }
\usepackage[normalem]{ulem}
\usepackage{hyperref}\hypersetup{colorlinks=true, linkcolor=BrickRed, citecolor=RoyalBlue, urlcolor=Sepia}
\usepackage{enumitem}
\usepackage{mathtools}
\usepackage[dvipsnames]{xcolor}
\usepackage{ stmaryrd }
\usepackage{cleveref}
\usepackage{xurl}
\numberwithin{equation}{section}

\newtheoremstyle{mytheorem}
{10pt}
{10pt}
{\it}
{\parindent}
{\bf}
{.}
{ }
{\thmnumber{#2.~}\thmname{#1}\thmnote{~\rm#3}}

\newtheoremstyle{myremark}
{10pt}
{10pt}
{\rm}
{\parindent}
{\bf}
{}
{ }
{\thmnumber{#2.~}\thmname{#1}\thmnote{~\rm#3}}

\newtheoremstyle{myparagraph}
{10pt}
{10pt}
{\rm}
{\parindent}
{\bf}
{.}
{ }
{\thmnumber{#2.~}\thmname{#1}\thmnote{#3}}

\theoremstyle{mytheorem}
\newtheorem{theorem}[subsubsection]{Theorem}
\newtheorem{definition}[subsubsection]{Definition}
\newtheorem{lemma}[subsubsection]{Lemma}
\newtheorem{corollary}[subsubsection]{Corollary}
\newtheorem{proposition}[subsubsection]{Proposition}

\theoremstyle{myremark}
\newtheorem{remark}[subsubsection]{Remark.}

\theoremstyle{myparagraph}
\newtheorem{parag}[subsubsection]{}
\newtheorem*{parag*}{}

\makeatletter
    \def\@secnumfont{\sc}
    \def\section{\@startsection%
    {section}
    {1}
    \z@{1.5\linespacing\@plus .2\linespacing}
      {.7\linespacing}
      {\normalfont\sc\centering}}
    \def\subsection{\@startsection{subsection}{2}
    \z@{1.0\linespacing\@plus .2\linespacing}
      {.5\linespacing}
      {\normalfont\it\centering}}
\makeatother

\makeatletter
\renewenvironment{proof}[1][\proofname]{\par 
    \pushQED{\qed}%
    \normalfont \topsep10\p@\@plus6\p@\relax 
    \trivlist 
    \item[\hskip\labelsep 
    \bfseries 
    #1\@addpunct{.}]\ignorespaces}
    {\popQED\endtrivlist\@endpefalse} 
\providecommand{\proofname}{Proof}
\makeatother

\newlist{enumeraterem}{enumerate}{1}
\setlist[enumeraterem]{label=(\roman*), leftmargin=0pt, itemsep=3pt, itemindent=30pt}
\newlist{enumeratethm}{enumerate}{1}
\setlist[enumeratethm]{label={\rm(\roman*)}, leftmargin=30pt, itemsep=2pt}

\newcounter{cnstcnt}
\newcommand{\newr}{%
\refstepcounter{cnstcnt}%
\ensuremath{r_{\thecnstcnt}}}
\newcommand{\oldr}[1]{\ensuremath{r_{\ref{#1}}}}

\newcounter{const}
\newcommand{\newC}{\refstepcounter{const}\ensuremath{C_{\theconst}}}
\newcommand{\oldC}[1]{\ensuremath{C_{\ref{#1}}}}
\newcommand{\bfi}{\mathbf{i}}
\newcommand{\Haus}{\mathscr{H}}
\newcommand{\Leb}{\mathscr{L}}

\newcommand{\Gr}{\mathrm{Gr}}

\newcommand{\Span}{\mathrm{span}}
\newcommand{\Tan}{{\rm Tan}}
\newcommand{\Exp}{{\rm Exp}}
\newcommand{\de}{\mathrm{d}}
\newcommand{\bd}{\partial}

\newcommand{\Lip}{{\rm Lip}}
\newcommand{\cc}{Carnot-Carath\'eodory }
\DeclareMathOperator{\esterno}{\mbox{\large$\wedge$}}
\DeclareMathOperator{\trace}{\mbox{\Large$\llcorner$}}
\DeclareMathOperator{\antitrace}{\mbox{\kern-1.5pt\Large$\lrcorner$\kern.5pt}}
\newcommand{\scalar}[2]{\langle #1;\, #2\rangle}

\newcommand{\N}{\mathbb{N}}

\newcommand{\Q}{\mathbb{Q}}
\newcommand{\R}{\mathbb{R}}

\begin{document}


\thispagestyle{empty}
~\vskip -1.1 cm

	%
	%

\vspace{1.7 cm}

	%
	%
{\centering\Large\bf
Tangency sets of non-involutive distributions

\smallskip
and unrectifiability in Carnot-Carath\'eodory spaces
}

\vspace{.7 cm}

	%
	%
{\centering\sc 
Giovanni Alberti, Annalisa Massaccesi, Andrea Merlo
\\
}

\vspace{.8 cm}

	%
	%
{\rightskip 1 cm
\leftskip 1 cm
\parindent 0 pt
\footnotesize
{\sc Abstract.} In this paper, we establish refined versions of the Frobenius Theorem for non-involutive distributions and use these refinements to prove an unrectifiability result for Carnot–Carathéodory spaces. We also introduce a new class of metric spaces that extends the framework of Carnot–Carathéodory geometry and show that, within this class, Carnot–Carathéodory spaces are, in some sense, extremal. Our results provide new insights into the relationship between integrability, non-involutivity, and rectifiability in both classical and sub-Riemannian settings.


\par
}

\tableofcontents

\section{Introduction}
\label{sec:intro}

Involutivity is a concept deeply connected to rectifiability. A classical cornerstone of differential geometry is the Frobenius Theorem, which implies that if \(V\) is a distribution of \(k\)-dimensional planes on an open set \(\Omega \subset \R^n\), and if \(\Sigma\) is a \(k\)-dimensional smooth surface that is everywhere tangent to \(V\), then \(V\) must be involutive at every point of \(\Sigma\).  
Though classical, this is a powerful, very rigid connection between integrability—that is, the existence of a surface tangent to a given distribution—and involutivity, which is an algebraic condition on the derivatives of the distribution or differential relation.

On the other hand, a Lusin-type theorem proved by the first-named author in \cite{zbMATH00021945} shows that there exist \(C^1\)-regular surfaces tangent to non-involutive distributions on sets of positive surface measure. This result was later refined by Z. Balogh in \cite{Balogh2003SizeGradient}, who proved that such rectifiable sets can be taken to be of class \(\bigcap_{0<\alpha<1}C^{1,\alpha}\). This regularity is sharp, since the classical Frobenius theorem shows that if the surface were \(C^{1,1}\), then by the classical Lusin theorem the tangency set would be null.

{In the same spirit, one could also wonder if the same duality between the rigidity of Frobenius-type theorems and the flexibility of Lusin-type constructions applies in other frameworks, such as the realm of currents. Indeed, currents can be interpreted as weak surfaces with possible rectifiability properties. The answer is positive, as one can see in \cite{AMFrob_lincei,alberti2020geometric}.}

In recent years there has been a surge of interest in Carnot–Carathéodory spaces that arise naturally in the theory of partial differential equations and harmonic analysis, with particular attention to Carnot groups, which serve as the local model for these spaces. {Let us recall that a Carnot–Carathéodory space is intrinsically endowed with a non-involutive distribution, on which its intrinsic metric is based.}

The study of geometric measure theory and rectifiability in Carnot groups was pioneered by the works of Ambrosio–Kirchheim \cite{AK00} and Franchi–Serapioni–Serra Cassano \cite{Serapioni2001RectifiabilityGroup}. Research in this area has been very active over the last two decades; see, e.g., \cite{MagnaniUnrect, step2, ambled, Mag13, MagnaniTowardArea, MatSerSC, JNGV20, DLDMV19, Vittone20, DDFO20, antonelli2020rectifiable2, antonelli2020rectifiableA, antonelli2020rectifiableB, Merlo_2021, MarstrandMattila20}.

One of the most intriguing aspects of Carnot groups is that they do not contain any low-codimensional Euclidean rectifiable sets. This is a well-known consequence of Pansu's differentiability theorem (see \cite{Pansu}). The main reason is that the Carnot–Carathéodory ball is forced to be squeezed in directions not tangent to the horizontal distribution, thereby preventing Lipschitz images of \(\R^m\) from spreading in those directions. As a result, Lipschitz images are forced to concentrate along the horizontal distribution, and this, on the one hand, gives an upper bound on the dimension of the images and, on the other, implies that they are \(C^2\)-rectifiable.

It is worth noting that in Carnot groups other notions of rectifiability are available (see \cite{Serapioni2001RectifiabilityGroup, MerloAntonelli,Merlo_2021,MarstrandMattila20, JNGV20, Pauls04}). These notions have almost nothing to do with the notion of Lipschitz continuity in the broadest sense of the word; however, it is possible to prove that Lipschitz images of Euclidean spaces are still the correct notion of regular surface in these very rough contexts (see \cite{AM20, IMM20, zbMATH07682699}).

This paper is devoted to highlighting the connections between the absence of rectifiable sets in sub-Riemannian geometry and the classical Frobenius Theorem.

\subsection{Main results}

This paper is ideally divided into three parts. The first part is devoted to refining the Frobenius Theorem for highly non-involutive distributions. We say that a \(k\)-dimensional distribution of planes is \emph{\(h\)-non-involutive} if every \(h\)-dimensional sub-distribution is non-involutive. The first result we obtain is of classical flavour.

\begin{parag}[Theorem. (Structure of tangency sets to \(C^{2}\) surfaces)]\label{th:main1.intro}
    Let \(2\leq h\leq k'\leq k<n\) and suppose \(V\) is an \(h\)-non-involutive \(k\)-dimensional distribution of class \(C^1\) in \(\R^n\) (in the sense of Definition \ref{def:hnoninv}). Let \(S\) be a submanifold of \(\R^n\) of dimension \(k'\) and of class \(C^{2}\). Then the set
    \[
    \mathscr{C}(S,V):=\{q\in S : \Tan(S,q) \subseteq V(q)\}
    \]
    is \((h-1)\)-rectifiable, where \(\Tan(S,q)\) denotes the classical tangent to the surface \(S\) at \(q\).
\end{parag}

Notice that the improvement on the classical Frobenius Theorem is twofold. If we do not impose any further hypothesis on \(V\) and we take \(h=k\), we see that we can infer more structure on the contact set: not only is it \(\Haus^k\)-null, but it is also \((k-1)\)-rectifiable. Secondly, if we have finer information on the distribution \(V\), i.e., that it is also \(h\)-non-involutive, then we infer that the tangency set is \((h-1)\)-rectifiable. Theorem \ref{th:main1.intro}, however, can be obtained with fairly standard techniques, and for our applications it is not strong enough. We were able to obtain the following result:

\begin{parag}[Theorem. (Structure of tangency sets to \(C^{1,1}\) surfaces)]\label{th:main2.intro}
    Let \(2\leq h\leq k'\leq k<n\) and suppose \(V\) is an \(h\)-non-involutive \(k\)-dimensional distribution of class \(C^1\) in \(\R^n\). Let \(S\) be a submanifold of \(\R^n\) of class \(C^{1,1}\) and of dimension \(k'\). Then \(\mathscr{C}(S,V)\) is \(h\)-purely unrectifiable.
\end{parag}

The above result trades off the regularity of the surface, passing from \(C^{2}\)-regularity to \(C^{1,1}\) while losing the strong geometric information of Theorem \ref{th:main1.intro}. This second result is way more delicate than the previous one. The attentive reader might wonder if it is possible to improve on \eqref{th:main2.intro}; however, unfortunately, it is not possible. The reason is the following: with the same arguments employed by the first-named author in \cite{zbMATH00021945}, for every \(h\)-non-involutive \(k\)-distribution \(V\) and every \(s<n\) it is possible to construct surfaces of class \(C^{1,1}\) for which the tangency set \(\mathscr{C}(S,V)\) has Hausdorff dimension \(s\). This should signal that passing from \(C^2\) to \(C^{1,1}\) represents a very delicate endeavour, also in light of the results obtained by Z. Balogh in \cite{Balogh2003SizeGradient}.

An application of Theorem \ref{th:main2.intro} is the following structural result for \cc spaces.

\begin{theorem}\label{TH:UN.intro}
Suppose \(V\) is a smooth distribution of \(k\)-planes with the H\"ormander condition (see \S\ref{par:hor}). Let \(1<h\leq k\) be the smallest positive integer for which \(V\) is \(h\)-non-involutive. Suppose \(K\) is a compact subset of \(\R^m\) for some \(m\geq h\) and \(f\) is a Lipschitz map from \(K\), endowed with the Euclidean distance, to the \cc space \((\R^n,d_V)\), where \(d_V\) is the natural \cc metric induced by the distribution \(V\). Then
$\Haus^{m}_{d_V}(f(K))=0$.
\end{theorem}

Theorem \ref{TH:UN.intro} not only tells us that all \cc spaces are \(k\)-purely unrectifiable, where we recall that \(k\) is the dimension of the horizontal distribution \(V\), but if the distribution is \(h\)-non-involutive then the \cc space \((\R^n,d_V)\) does not contain any non-null \(h\)-rectifiable sets. Besides extending to all \cc spaces the well-known results of pure unrectifiability of Carnot groups (see, e.g., \cite{AK00,MagnaniUnrect}), the strategy of the proof of Theorem \ref{TH:UN.intro} is also completely novel. For the sake of simplicity we will stick to discussing the proof in the case where the distribution \(V\) is of step 2, i.e., if the H\"ormander condition is satisfied with just the first commutators of the vectors spanning \(V\). Thanks to the Ball-Box lemma (see Proposition \ref{BallBox}), we prove that Lipschitz images of compact sets of \(\R^m\) must have an \emph{osculating paraboloid}. This shows that these images are contained in the countable union of \(C^2\) \(k\)-dimensional surfaces. By Theorem \ref{th:main2.intro} this implies that these sets are \(\Haus^m\)-null, where here \(\Haus^m\) stands for the Euclidean Hausdorff measure. In order to promote this nullness to the \emph{intrinsic Hausdorff measure} \(\Haus^m_{d_V}\), we will perform a delicate cut of the coverings of the sets for which the \(\Haus^m\) premeasure is small (see the proof of Theorem \ref{TH:UN}).

In order to complete our analysis of the relationship between non-involutivity and \cc spaces, we introduce the concept of an \(\eta\)-squeezed metric. Let \(V\) be a distribution of \(k\)-planes in \(\R^n\). A metric \(d\) on \(\R^n\) is said to be \(\eta\)-squeezed, with \(\eta\in [1,2]\), if the metric balls \(B(x,r)\) (relative to \(d\)) are contained in cylinders with a base of radius \(r>0\) and height of order \(r^\eta\) (see Definition \ref{def:squmet} for a precise definition). This constraint on the shape of the ball is a substitute for the Ball-Box lemma, and indeed we are able to prove that all step-2 distributions are \(2\)-squeezed (see Proposition \ref{inclusioncc}). For a fixed \(V\) and \(\eta\), these metrics are all locally bi-Lipschitz equivalent, and we build a variant of the \cc distance (see Definition \ref{dVmtr}) that is \(\eta\)-squeezed. Moreover, we prove that the dependence on \(\eta\) is continuous (see \cref{limit}). Interestingly, then the following result holds.

\begin{proposition}
If the metric \(d\) is \(\eta\)-squeezed on \(V\) for some \(\eta\in (1,2)\), then there exists a Lipschitz function $f:K\to (\mathbb{R}^n,d)$
with \(K\Subset \mathbb{R}^k\), such that
\[
\Haus^k_{d}\bigl(f(K)\bigr)>0.
\]
On the other hand, if \(\eta=2\) and \(V\) is non-involutive, then the metric space \((\mathbb{R}^n,d)\) is \(k\)-purely unrectifiable.
\end{proposition}

This proposition, together with the above discussion, shows that \cc spaces are extremal in the following sense. As remarked above, we show that \cc spaces of step 2 are Gromov-Hausdorff limits of metric \(\eta\)-squeezed metric spaces that contain \(k\)-rectifiable sets (where \(k\) is the dimension of \(V\)), and whose balls are cylinders with a base of radius \(r\) and height \(r^\eta\). While these spaces contain \(k\)-rectifiable sets for every \(\eta<2\), in the critical case \(\eta=2\) (with \(V\) non-involutive) these rectifiable objects disappear in the limit. This highlights the fact that the exponent \(2\) in the Ball-Box lemma is crucial in obtaining the geometric properties of \cc spaces.

{\small
\begin{parag*}[Acknowledgements]
Part of this research was carried out during visits
of the second and third author at the Mathematics Department in Pisa.
These were supported by the University of Pisa through 
the 2015 PRA Grant ``Variational methods for geometric problems'', by the 2018 INdAM-GNAMPA 
project ``Geometric Measure Theoretical approaches to Optimal Networks'' and by PRIN project MUR 2022PJ9EFL ``Geometric Measure Theory: Structure of Singular Measures, Regularity Theory and Applications in the Calculus of Variations''.
   
The research of G.A. has been partially supported
by the Italian Ministry of Education, University and Research (MIUR) 
through the PRIN project 2010A2TFX2
The research of A. Ma has received funding from the European Union's Horizon 2020 research and innovation programme under the grant agreement No. 752018 (CuMiN), STARS@unipd research grant ``QuASAR -- Questions About Structure And Regularity of currents'' (MASS STARS MUR22 01) and INdAM project ``VAC\&GMT''.  The research of A. Me has has received funding from the European Union's Horizon 2020 research and innovation programme under the Marie Sk\l odowska-Curie grant agreement no 101065346.
\end{parag*}
}

\section{Notations and preliminaries}

\subsection{Notation}

In this section we need introduce some frequent notations used throughout the paper.

\begin{itemize}
[leftmargin=50pt, itemsep=4 pt plus 1pt, labelsep=10pt]
\item[$I(n,k)$]
set of all multi-indices $\bfi:=(i_1,\dots,i_k)$ 
with $1\le i_1<\dots<i_k\le n$;

\item[$\esterno_k(V)$]
space of $k$-vectors in a linear space $V$;  
the canonical basis of $\esterno_k(\R^n)$ is formed by the 
simple $k$-vectors $e_\bfi:=e_{i_1}\wedge\dots\wedge e_{i_k}$ 
with $\bfi\in I(n,k)$, where $\{e_1,\dots,e_n\}$
is the canonical basis of $\R^n$;
$\esterno_k(\R^n)$ is endowed with the Euclidean
norm $|\cdot|$ associated to this basis;

\item[$\esterno^k(V)$]
space of $k$-covectors on a linear space $V$; 
the canonical basis of $\esterno^k(\R^n)$ is formed by the 
simple $k$-covectors 
$\de x_\bfi:=\de x_{i_1}\wedge\dots\wedge \de x_{i_k}$ 
with $\bfi\in I(n,k)$, where $\{\de x_1,\dots, \de x_n\}$
is the canonical basis of the dual of $\R^n$;
$\esterno^k(\R^n)$ is endowed with Euclidean norm $|\cdot|$
associated to this basis;

\item[$\de x$]
$:= \de x_1\wedge\dots\wedge \de x_n$;

\item[$\wedge$]
exterior product of $k$-vectors, or of $h$-covectors;

\item[$ \trace$]

interior products of a $k$-vector and a $h$-covector
(\S\ref{def:int_prod});

\item[$\de$]
exterior derivative of a $k$-form (\S\ref{def:diffop}); 

\item[$\mathrm{div}$]
divergence of a $1$-vectorfield (\S\ref{def:diffop});

\item[{$[v,v']$}]
Lie bracket of vector fields $v$ and $v'$ (\S\ref{def:lie});

\item[{$\Haus^m$}] $m$-dimensional Hausdorff measure;

\item[{$\mathrm{Exp}_x$}] Exponential map based at $x$ (\S\ref{par:BBT});
\end{itemize}

Throughout this paper, we adopt Federer's convention, denoting closed balls by \(B(x,r)\) and open balls by \(U(x,r)\) in a metric space.

\subsection{Multilinear algebra and differential calculus of forms}

\begin{parag}[Interior product]
\label{def:int_prod}
Given a $k$-vector $v$ in $V$ and an $h$-covector 
$\alpha$ on $V$ with $h\le k$, the \emph{interior product} 
$v\trace\alpha$ is the $(k-h)$-vector in $V$ defined by
\[
\scalar{v\trace\alpha}{\beta}:=\scalar{v}{\alpha\wedge\beta}
\quad\text{for every $(k-h)$-covector $\beta$.}
\]
\end{parag}

\begin{parag}[Forms]
Let $k\in\{1,\ldots,n-1\}$.
We view $k$-forms as maps $\omega:\R^n\to\esterno^k(\R^n)$, 
which we sometime write in terms of the canonical basis of
$\esterno^k(\R^n)$, that is
\[
\omega(x) = \sum_{\bfi\in I(n,k)} \omega_\bfi(x) \, \de x_\bfi
\, .
\]
\end{parag}

\begin{parag}[Distributions of $\boldsymbol k$-planes]
\label{def:distrib}
A \emph{$k$-dimensional distribution of planes} in $\R^n$ is a map  $V:\R^n\to \Gr(n,k)$, where as usual with the symbol $\Gr(n,k)$ we denote the Grassmannian of $k$-dimensional planes of $\R^n$. 
We say that a \emph{system of vector fields}  $\mathcal{X}:=\{X_1,\dots, X_k\}$, where $k\in \N$,
\emph{spans} the distribution $V$ if, for every $x\in \R^n$, one has
\[
V(x)
=\Span(\mathcal{X}(x)) :=\Span\big\{ X_1(x),\dots,X_k(x) \big\}
\, .
\] 
Let $r\in \N$. The distribution $V$ is said to be of class $C^r$ if there is a system of vector fields of class $C^r$ that span $V(x)$ at every $x\in \R^n$. Finally, we say that a vector field $Y$ is tangent to the distribution $V$ if $Y(x)\in V(x)$ for every $x\in \R^n$ and if $\mathcal{W}$ is a system of vector fields, we say that $\mathcal{W}$ is tangent to $V$ if each of its vector fields is tangent to $V$.
\end{parag}


\begin{parag}[Restriction of forms to distributions]
    Let \( k \leq \ell \) and let \( W \) be a distribution of \(\ell\)-planes. For every \( k \)-form \(\omega\) on \( M \), we define the \emph{restriction} of \(\omega\) to \( W \), denoted by \(\omega|_W\), as the \( k \)-form on \( W \) given at each point \( x \in M \) by
\[
(\omega|_W)_x(v_1, \dots, v_k) := \omega_x(v_1, \dots, v_k)\qquad\text{for all \( v_1, \dots, v_k \in W(x) \).}
\]
\end{parag}

\begin{parag}[Distributions as the kernel of forms] 
Let $r\in \N$ and $V$ be a $k$-dimensional distribution in $\R^n$ of class $C^r$. It is immediate to see that $V(x)$ can be identified with the intersection of the kernel of $(n-k)$ independent $1$-forms $\omega_1,\ldots,\omega_{n-k}$ of class $C^r$.   
\end{parag}

\begin{parag}[Lie brackets of vector fields]
\label{def:lie}
Recall that, given two vector fields $X$, $X'$ on $\R^n$ of class $C^1$, 
the Lie bracket $[X,X']$ is the vector field on $\R^n$ 
defined by
\[
[X,X'](x)
:= \frac{\bd X}{\bd X'}(x) - \frac{\bd X'}{\bd X}(x)
= \mathrm{d}_x X \, (X'(x)) - \mathrm{d}_x X' \, (X(x))
\, , 
\]
where $\mathrm{d}_x X$ and $\mathrm{d}_x X'$ stand for the differentials of $X$ and $X'$
at the point $x$, viewed as linear maps from $\R^n$ into itself.

    Given a system of $k$ smooth vector fields $\mathcal{X}:=\{X_1,\dots, X_k\}$ on $\R^n$, we say that a vector field $Y$ is an \emph{elementary commutator of $\mathcal{X}$} if there are  $X_{i_1},\ldots,Y_{i_N}\in\mathcal{X}$ such that
    \begin{equation}
        Y=[X_{i_1},[\ldots,[X_{i_{N-1}},Y_{i_N}]\ldots]].
        \nonumber
    \end{equation}
In addition, if a vector field $Y$ is an elementary commutator we will denote by $\mathrm{deg}(Y)$, \emph{the degree of $Y$}, the number
$$\min\big\{N:\text{there are }Y_1,\ldots,Y_N\in\mathcal{X}\text{ such that } Y=[X_{i_1},[\ldots,[X_{i_{N-1}},Y_{i_N}]\ldots]]\big\}.$$
\end{parag}

\begin{parag}[Exterior derivative and divergence]
\label{def:diffop}
If $\omega$ is a $k$-form of class $C^1$, 
the \emph{exterior derivative} $\de\omega$ is the $(k+1)$-form 
defined in coordinates by the usual formula:
\begin{equation}
\label{e:diff}
\de\omega(x)
:=\sum_{\bfi\in I(n,k)} \sum_{j=1}^n
   \frac{\bd\omega_\bfi}{\bd x_j}(x) \, \de x_j\wedge\de x_\bfi
= \sum_{j=1}^n \de x_j \wedge \frac{\bd\omega}{\bd x_j}(x)
\, .
\end{equation}
Given a vector field $X$ of class $C^1$ on $\R^n$ we define as usual
$$\mathrm{div}X:=\sum_{i=1}^n\frac{\partial X_i}{\partial x_i}.$$
\end{parag}

\subsection{Notions of involutivity}

\begin{parag}[Non-involutivity of a $\boldsymbol{k}$-dimensional distribution $\boldsymbol{V}$]
\label{def:inv}
Let $V$ be a distribution of planes
of class $C^1$ in $\R^n$.
We say that $V$ is \emph{non-involutive} at a point $x$ if
there exists a couple of vector fields $X$, $X'$ of class $C^1$ which
are tangent to $V$ whose
the commutator $[X,X'](x)$ \emph{does not belongs to $V(x)$}.
We say that $V$ is \emph{non-involutive} (respectively \emph{involutive}) if it is non-involutive (respectively involutive) at every point of $\R^n$.
\end{parag}

\begin{lemma}
  \label{prop:iddifferenziale}
    Let $\omega$ be a $1$-form and $X,Y$ be vector fields. Then 
$$-\mathrm{div}(X\wedge Y\trace \omega)=\mathrm{d}\omega(X\wedge Y)+\omega(X)\mathrm{div}(Y)-\omega(Y)\mathrm{div}(X)-\omega([X,Y]).$$
\end{lemma}

\begin{proof}
     The action of $d\omega$ on the $2$-vector $X\wedge Y$ can be written in coordinates as 
$$\mathrm{d}\omega(X\wedge Y)=\sum_{ i,j}(\partial_i\omega_j-\partial_j\omega_i)X_iY_j.$$
Leibniz's rule tells us that 
\begin{equation}
    \begin{split}
\partial_i(\omega_jX_iY_j)=\partial_i\omega_jX_iY_j+\omega_j\partial_iX_iY_j+\omega_jX_i\partial_iY_j,\\
\partial_j(\omega_iX_iY_j)=\partial_j\omega_iX_iY_j+\omega_i\partial_jX_iY_j+\omega_iX_i\partial_jY_j.
    \end{split}
\end{equation}
Subtracting the two expressions above and summing them over the indexes $i,j$, we infer that 
\begin{equation}
    \begin{split}
        \sum_{i,j} \partial_i(\omega_jX_iY_j)-&\partial_j(\omega_iX_iY_j)=d\omega(X\wedge Y)+\\
        +&\sum_{i,j}(\omega_j\partial_iX_iY_j-\omega_iX_i\partial_jY_j)+\sum_{i,j}(\omega_jX_i\partial_iY_j-\omega_i\partial_jX_iY_j)\\
        =&\mathrm{d}\omega(X\wedge Y)+\omega(X)\mathrm{div}(Y)-\omega(Y)\mathrm{div}(X)-\omega([X,Y]).
        \label{id:diff1}
    \end{split}
\end{equation}
Finally, we are left to study the left-hand side of \eqref{id:diff1}. Note that 
\begin{equation}
    \begin{split}
        \sum_{i,j} \partial_i(\omega_jX_iY_j)-\partial_j(\omega_iX_iY_j)=\sum_{i,j} \partial_i(\omega_j(X_iY_j-X_jY_i))=-\mathrm{div}(X\wedge Y\trace \omega).
        \nonumber
    \end{split}
\end{equation}
This concludes the proof.
\end{proof}

\begin{proposition}\label{prop:iddifferenziale:utile}
    Let $\omega$ be a $1$-form of class $C^1$ and $X,Y$ be vector fields of class $C^1$ such that $\omega(X)=\omega(Y)=0$. Then 
$$\mathrm{d}\omega(X\wedge Y)=\omega([X,Y]).$$
\end{proposition}

\begin{proof}
    Since $X,Y$ are in the kernel of $\omega$, Proposition \ref{prop:iddifferenziale} implies that 
$$-\mathrm{div}(X\wedge Y\trace \omega)=\mathrm{d}\omega(X\wedge Y)-\omega([X,Y]).$$
However, notice further that $X\wedge Y$ is a $2$-vector tangent to the $1$-codimensional distribution $\mathrm{Ker}(\omega)$, i.e. $X\wedge Y(z)$ is a $2$-vector of $\mathrm{Ker}(\omega(z))$ for any $z\in \R^n$. This implies that 
$$X\wedge Y\llcorner \omega=\omega(X)Y-\omega(Y)X=0.$$
The above identity concludes the proof of the proposition. 
\end{proof}

\begin{parag}[Commutators of weighted vector fields]\label{commexpr}
Let $X,Y$ be vector fields of class $C^1$ and let $f,g$ be two functions of class $C^1$. Then, in coordinates we have 
    \begin{equation}
        \begin{split}
           [fX,gY]=& \sum_{i,j=1}^n (fX_j\partial_j(gY_i)-gY_j\partial_j(fX_i))e_i\\
           =&fg[X,Y]+f\langle \nabla g,X\rangle Y-g\langle\nabla f,Y\rangle X.
           \nonumber
        \end{split}
    \end{equation}
\end{parag}

\begin{proposition}\label{prop:implicinv}
Let $V$ be a distribution of $k$-planes in $\R^n$. The following are equivalent
\begin{itemize}
    \item[(i)]$V$ is  non-involutive at $x\in \R^n$;
    \item[(ii)]there exists a ball $B$ centered at $x$ and a $1$-form $\omega$ of class $C^1$ such that $\omega\rvert_{V}=0$ for any $z\in B$ and $\mathrm{d}\omega\rvert_{V}(x)\neq 0$;
    \item[(iii)] let $Y_1,\ldots, Y_k$ be vector fields of class $C^1$ spanning $W$ in a neighborhood of $x$. Then, there are $a_1,a_2\in \{1,\ldots,k\}$ such that $[Y_{a_1},Y_{a_2}](x)\not\in V(x)$.
\end{itemize} 
\end{proposition}

\begin{proof}
Suppose $V$ is non-involutive at $x$, there are two vector fields tangent to $V$ such that $[X,Y](x)\not\in V(x)$. This implies in particular that there exists a $1$-form of class $C^1$ such that $\omega\vert_{V}=0$ and $\omega_x([X,Y](x))\neq 0$. Proposition \ref{prop:iddifferenziale:utile} implies in particular that at $x$, we have
$$(\mathrm{d}\omega)_x(X\wedge Y(x))=\omega_x([X,Y](x))\neq 0.$$
Viceversa, suppose there exists a ball $B$ centered at $x$ and a $1$-form $\omega$ of class $C^1$ such that $\omega\vert_{V(z)}=0$ for every $z\in B$ and $(\mathrm{d}\omega_x)\lvert_{V(x)}\neq 0$. Then, there is $w\in \Lambda^2(V(x))$ such that $(\mathrm{d}\omega_x)(w)\neq 0$. Denoted by $\{v_1,\ldots, v_k\}$ a family of independent vectors of $V(x)$,
there are coefficients $\lambda_{i,j}\in \R$ with $i,j\in\{1,\ldots, k\}$ such that 
$$w=\sum_{i<j}\lambda_{i,j}v_i\wedge v_j.$$
Since $\{v_i\wedge v_j\}_{i<j}$ is a basis of $\Lambda^2(V(x))$, we infer that there must exist $i<j$ such that 
$$\mathrm{d}\omega_x(v_i\wedge v_j)\neq 0.$$
Let $X_1,\ldots,X_k$ be $C^1$ vector fields tangent to $V$ such that $X_i(x)=v_i$ for every $i=1,\ldots,k$. However, it is immediate to see, thanks to the very definition of $\mathrm{d}\omega$, that 
$$0\neq \mathrm{d}\omega_x(v_i\wedge v_j)=\mathrm{d}\omega(X_i\wedge X_j)(x).$$
This proves the equivalence of (i) and (ii).

The fact that (i) implies (iii) is immediate while the viceversa is more delicate. Suppose by contradiction (iii) did not imply (i). Since $V$ is supposed to be non-involutive we can find two vector fields $X_1,X_2$ of class $C^1$ such that $[X_1,X_2](x)\not \in V(x)$. It is immediate to see that there are functions $\eta_k^j$ with $k\in \{1,2\}$ and $j\in \{1,\ldots,n\}$ of class $C^1$ such that $$X_k=\sum_{j=1}^n\eta_k^jY_j.$$
    Thank to \S\ref{commexpr}, we have 
    \begin{equation}
        \begin{split}
            &\qquad\qquad\qquad [X_1,X_2](x)=\sum_{j_1,j_2}[\eta_1^{j_1}Y_{j_1},\eta_2^{j_2}Y_{j_2}](x)\\
    =&\sum_{j_1,j_2}(\eta_1^{j_1}\eta_2^{j_2}[Y_{j_1},Y_{j_2}])(x)+\sum_{j_1,j_2}(\eta_1^{j_1}\langle \nabla \eta_2^{j_2},Y_{j_1}\rangle Y_{j_2}-\eta_2^{j_2}\langle \nabla \eta_1^{j_1},Y_{j_2}\rangle Y_{j_1})(x).
    \nonumber
        \end{split}
    \end{equation}
However, since for every $j_1,j_2\in \{1,\ldots,n\}$ we have $[Y_{j_1},Y_{j_2}](x)\in V(x)$ and $Y_{j_1}(x),Y_{j_2}(x)\in V(x)$ we see how the above identity chain results in contradiction with the fact that $[X_1,X_2](x)\not\in V(x)$. 
\end{proof}

\begin{parag}[$h$-non-involutivity]\label{def:hnoninv}
    Let $V$ be a $C^1$ distribution of $k$-planes and $h\leq k$. We say that $V$ is \emph{$h$-non-involutive} at a point $x$ if any $C^1$ regular $h$-dimensional distribution tangent to $V$ is non-involutive at $x$.
\end{parag}

\begin{parag}[Equivalence of $C^k$ and $C^1$ $h$-non-involutivity]
    Let $h\leq k$ and suppose $V$ is a $k$-dimensional distribution of class $C^\ell$, for some $\ell\geq 1$. Clearly, if $V$ is $h$-non-involutive at $x$ then for every $\ell\leq h$ each $C^\ell$-regular $h$-dimensional distribution tangent to $V$ is non-involutive at $x$. 
    
    Viceversa, suppose that every $C^\ell$-regular distribution tangent to $V$ of dimension $h$ is non-involutive at $x$. 
    Let $X_1,\ldots X_k$ be vector fields of class $C^\ell$ spanning $V$ and note that for any $h$-distribution $W$ of class $C^1$ tangent to $V$, we can find $C^1$ functions $p_{i,j}$ with $i\in \{1,\ldots,h\}$ and $j\in\{1,\ldots,k\}$ such that the vector fields $Y_i:=\sum_{j=1}^kp_{i,j}X_j$ span $W$. 
        We claim that there are $a,b\in\{1,\ldots,h\}$ such that $[Y_a,Y_b](x)\not\in W(x)$.

Let $q_{i,j}(y):=p_{i,j}(x)+\nabla p_{i,j}(x)[y-x]$,
and denote by $\bar{Y}_i$ the vector field 
$$\bar{Y}_{i}:=\sum_{i,j} q_{i,j}X_j.$$
The vector fields $\bar{Y}_i$ are of class $C^r$, and 
\begin{equation}
    \begin{split}
        [\bar Y_a,\bar Y_b]=&\sum_{i,j}q_{a,i}q_{b,j}[X_i,X_j]+\sum_{i,j}(q_{a,i}\langle \nabla q_{b,j},X_i\rangle X_j-
q_{b,j}\langle \nabla q_{a,i},X_j\rangle X_i)\\
=&\sum_{i,j}q_{a,i}q_{b,j}[X_i,X_j]+\sum_{i,j}(q_{a,i}\langle \nabla p_{b,j}(x),X_i\rangle X_j-
q_{b,j}\langle \nabla p_{a,i}(x),X_j\rangle X_i).
\nonumber
    \end{split}
\end{equation}
Moreover, at the point $x$, the vector fields $\bar Y_i$ span $W(x)$. This concludes the proof, indeed
\begin{equation}
    \begin{split}
        &\qquad\qquad\qquad[\bar Y_a,\bar Y_b](x)=\sum_{i,j}p_{a,i}(x)p_{b,j}(x)[X_i,X_j](x)\\
&+\sum_{i,j}(p_{a,i}(x)\langle \nabla p_{b,j}(x),X_i\rangle X_j-
p_{b,j}(x)\langle \nabla p_{a,i}(x),X_j\rangle X_i)=[Y_a,Y_b](x).
\nonumber
    \end{split}
\end{equation}
The vector fields $\bar Y_1,\ldots,\bar Y_h$ in a neighbourhood of $x$ are independent and thus they span a $h$-dimensional distribution $\bar W$ of class $C^r$, which thanks to our hypothesis is non-involutive. Therefore, thanks to Proposition \ref{prop:implicinv}, we can find $a,b\in\{1,\ldots,h\}$ such that $[Y_a,Y_b](x)=[\bar Y_a,\bar Y_b](x)\not\in W(x)$.   
\end{parag}

\begin{parag}[An example]
Thanks to Proposition \ref{prop:implicinv} it is immediate that a $k$-distribution $V$ is $h$-non-involutive if and only if for every $h$-dimensional sub-distribution $W$ of $V$ there exists a $1$-form $\omega$ satisfying 
\[
\omega\vert_W = 0 \quad \text{and} \quad \mathrm{d}\omega\vert_W(x) \neq 0.
\]

An elementary consequence of the above observation is the following: if a $k$-dimensional distribution $V$ has the property that for every $h$-dimensional sub-distribution $W$ there exists a $C^1$ $1$-form $\omega$ with $\omega\vert_V = 0$ and $\mathrm{d}\omega\vert_W(x) \neq 0$, then $V$ is $h$-non-involutive. A natural question is whether this sufficient condition is also necessary. Unfortunately, as we shall see below, the answer is no. However, in many concrete cases, such as the horizontal distribution of Carnot groups, this stronger form of $h$-non-involutivity is indeed satisfied.

We now proceed with the construction of a counterexample. Let $F(3,3)$ denote the free Lie algebra with three generators $e_1, e_2, e_3$ and three steps. Associated to $F(3,3)$ is the Carnot group $\mathbb{F}_{3,3}$, whose Lie algebra is $F(3,3)$. By identifying $\mathbb{F}_{3,3}$ with $\R^{14}$ (endowed with the group operation induced by the Baker--Campbell--Hausdorff formula), one can construct left-invariant vector fields $X_1,\ldots,X_{14}$ that, at the origin, coincide with the canonical basis $e_1,\ldots,e_{14}$ of $\R^{14}$ and satisfy the following commutation relations
\begin{equation}
    \begin{split}
        &[X_1,X_2]=X_4,\,\,\,[X_1,X_3]=X_5,\,\,\, [X_2,X_3]=X_6,
        [X_1,X_4]=X_7,\,\,\, [X_1,X_5]=X_8,\\
        &\qquad[X_1,X_6]=X_9,\,\,\,[X_2,X_4]=X_{10},\,\,\,[X_2,X_5]=X_{11},\,\,\,
        [X_2,X_6]=X_{12},\\
         &\qquad\qquad[X_3,X_4]=-X_8+X_{11},\,\,\, [X_3,X_5]=X_{13},\,\,\,[X_3,X_6]=X_{14}.
         \nonumber
    \end{split}
\end{equation}

Let $V$ be the smooth $6$-dimensional distribution spanned by the vector fields $\{X_1,X_2,\ldots,X_6\}$. Note that any $2$-dimensional sub-distribution of class $C^1$ of $V$ is non-involutive. Indeed, let $\alpha_1,\ldots,\alpha_6$ and $\beta_1,\ldots,\beta_6$ be $C^1$ functions, and define
\[
Y:=\sum_{i=1}^6\alpha_i X_i\quad \text{and}\quad Z:=\sum_{i=1}^6\beta_i X_i.
\]
One easily verifies that
\begin{equation}
    \begin{split}
[Y,Z] = \sum_{i,j=1}^6 \alpha_i\beta_j[X_i,X_j] 
&\;+\; \sum_{i,j=1}^6 \Bigl(\alpha_i\langle \nabla \beta_j,X_i\rangle X_j - \beta_j\langle \nabla \alpha_i,X_j\rangle X_i\Bigr).
    \end{split}
\end{equation}
An elementary computation shows that
\begin{equation}
    \begin{split}
&[Y,Z] + V = (\alpha_1\beta_4-\alpha_4\beta_1)X_7+(\alpha_1\beta_5-\alpha_5\beta_1)X_8+(\alpha_1\beta_6-\alpha_6\beta_1)X_9\\[1mm]
    &\qquad+ (\alpha_2\beta_4-\alpha_4\beta_2)X_{10}+(\alpha_2\beta_5-\alpha_5\beta_2)X_{11}+(\alpha_2\beta_6-\alpha_6\beta_2)X_{12}\\[1mm]
    + &(\alpha_3\beta_4-\alpha_4\beta_3)(-X_8+X_{11})+(\alpha_3\beta_5-\alpha_5\beta_3)X_{13}+ (\alpha_3\beta_6-\alpha_6\beta_3)X_{14} + V.
    \nonumber
    \end{split}
\end{equation}
The expression on the right-hand side implies that $[Y,Z]$ belongs to $V$ if and only if $Y$ and $Z$ are linearly dependent.

On the other hand, consider the distribution $W$ spanned by $X_1, X_2, X_3$. This distribution is clearly non-involutive. Nevertheless, for any $1$-form $\omega$ of class $C^1$ satisfying $\omega\vert_V=0$, we have $\mathrm{d}\omega\vert_W=0$. Indeed, for every pair of $C^1$ vector fields $Y,Z$ tangent to $W$ we have
\[
\mathrm{d}\omega(Y\wedge Z)=\omega([Y,Z])=0,
\]
since all the commutators of the vector fields $X_1, X_2, X_3$ lie in $V$.
\end{parag}

\subsection{A technical lemma}

\begin{parag}[Currents]
\label{def:curr}
A $k$-dimensional current ($k$-current)
$T$ on the open set $\Omega$ in $\R^n$ is a continuous 
linear functional on the space of smooth $k$-forms 
with compact support in~$\Omega$. 
The \emph{boundary} of $T$ is the $(k-1)$-current 
$\bd T$ on $\Omega$ defined by 
$\scalar{\bd T}{\omega} := \scalar{T}{d\omega}$
for every smooth $(k-1)$-form $\omega$ with compact support.
The \emph{mass} of $T$, denoted by 
$\mathbb{M}(T)$, is the supremum of $\scalar{T}{\omega}$ over
all $k$-forms $\omega$ such that $|\omega(x)|\le 1$ for every~$x\in\Omega$.

By Riesz's representation theorem, the fact that $T$ has finite mass 
is equivalent to saying that $T$ can be represented 
as a finite measure on $\Omega$ with values 
in the space $\esterno_k(\R^n)$, that is,
$T=\tau\mu$ where $\mu$ is a finite positive measure
on $\Omega$ and $\tau$ is a Borel $k$-vector field 
in~$L^1(\mu)$.
Thus
\[
\scalar{T}{\omega} 
= \int_{\Omega} \scalar{\tau(x)}{\omega(x)} \, d\mu(x)
\]
for every admissible $k$-form $\omega$ on $\Omega$,   
and $\mathbb{M}(T)=\int_{\Omega} |\tau| \, d\mu$.
Finally, a $k$-current $T$ is said to be \emph{normal} if both $T$ 
and $\bd T$ have finite mass.
\end{parag}

\begin{parag}[Distributional differential of $1$-forms]
    Let $\omega$ be a continuous $1$-form. We let $\mathrm{d}\omega$ be the distributional external differential of $\omega$, defined by duality as
    $$\langle\partial T,\omega\rangle=\langle T,d\omega\rangle,$$
for any $2$-current $T$ such that $T=\tau\mathscr{L}^n$, where $\tau$ is a $2$-vector of class $C^\infty_c$.  
\end{parag}

\begin{proposition}\label{lemmabordodistro}
    Let $\omega$ be a continuous $1$-form, such that its distributional external differential $\mathrm{d}\omega$ is represented by a continuous $2$-form. Then, 
$$\langle \partial T,\omega\rangle=\langle T,\mathrm{d}\omega\rangle\qquad\text{for any $2$-normal current $T$ with compact support.}$$ 
\end{proposition}

\begin{proof}
    Let $\rho_\varepsilon$ be a standard radial kernel of convolution and let $T_\varepsilon:=T*\rho_\varepsilon$. There holds that $\partial(T_\varepsilon)=(\partial T)_\varepsilon$, indeed 
    \begin{equation}
        \begin{split}
            \langle(\partial T)_\varepsilon,\tilde\omega\rangle=&\langle \partial T,\tilde\omega*\rho_\varepsilon\rangle\\
            =&\langle T,\mathrm{d}(\tilde\omega*\rho_\varepsilon)\rangle=\langle T,\mathrm{d}\tilde\omega*\rho_\varepsilon\rangle=\langle T_\varepsilon,\mathrm{d}\tilde\omega\rangle=\langle\partial(T_\varepsilon),\tilde\omega\rangle,
            \nonumber
        \end{split}
    \end{equation}
    for any $1$-form $\tilde\omega$ of class $C^\infty_c$. Note that in the above computations we exploited the fact that $\tilde{\rho}_\varepsilon(z)=\rho_\varepsilon(-z)$ for any $z\in \R^n$.  The action of $\partial(T_\varepsilon)$ and $(\partial T)_\varepsilon$ on $\omega$ and $\mathrm{d}\omega$ can be represented by integration. Therefore, the identity 
    $\langle T,\mathrm{d}\omega*\rho_\varepsilon\rangle=\langle\partial T,  \omega*\rho_\varepsilon\rangle$ follows from the chain of identities below
    $$\langle T,\mathrm{d}\omega*\rho_\varepsilon\rangle=\langle T_\varepsilon,\mathrm{d}\omega\rangle=\langle\partial(T_\varepsilon),\omega\rangle=\langle (\partial T)_\varepsilon,\omega\rangle=\langle \partial T,\omega*\rho_\varepsilon\rangle,$$
    where the first and last identity follow from the fact that $\partial T$ and $T$ are compactly supported measures. The second identity follows from the definition of distributional differential. The third identity follows by observing that the currents $\partial(T_\varepsilon)$ and $(\partial T)_\varepsilon$ can be represented as vector-valued compactly supported Radon measures. For the same reason, since $\omega*\rho_\varepsilon$ and $\mathrm{d}\omega*\rho_\varepsilon$ converge locally uniformly to $\omega$ and $\mathrm{d}\omega$ respectively, we see that $\langle \partial T,\omega\rangle=\langle T,\mathrm{d}\omega\rangle$ concluding the proof of the proposition. 
\end{proof}

\begin{proposition}\label{id:distromega}
    Let $\Omega\subseteq \R^k$ be an open set be such that  $\Phi\in\mathscr{C}^{1}(\Omega,\R^n)$ and let $\omega$ be a $1$-form of class $C^1$. Then
$$\mathrm{d}(\Phi^\#\omega)=\Phi^\#(\mathrm{d}\omega),$$
    where the identity above has to be understood in the sense of distributions. 
\end{proposition}

\begin{proof}
    First of all, let us note that the form $\tilde \omega$ is a Lipschitz form. Furthermore, it is immediate to see that $\Phi^\#(\mathrm{d}\omega)$ is a continuous $2$-form. Note that if $\Psi$ is a $C^2$ regular function and $\omega$ is a $1$-form of class $C^1$, we have
$$\mathrm{d}(\Psi^\#\omega)=\Psi^\#(\mathrm{d}\omega).$$
The identity above is classical and it holds pointwise. Now, let $\Psi_n$ be a sequence of $C^2$ functions converging to $\Psi$ in the topology induced by the $C^1$ norm. It is immediate to see that $\Psi_n^\#\omega$ and $\Psi^\#_n(\mathrm{d}\omega)$ converge uniformly to $\Phi^\#\omega$ and $\Phi^\#(\mathrm{d}\omega)$ respectively. For any $2$-current $T$ such that $T=\tau\mathscr{L}^n$, where $\tau$ is a $2$-vector of class $C^\infty_c$ we have 
$$\langle T,\Psi_n^\#(\mathrm{d}\omega)\rangle=\langle T,\mathrm{d}(\Psi_n^\#\omega)\rangle=\langle \partial T,\Psi_n^\#\omega\rangle.$$
However, since $T$ and $\partial T$ are measures, thanks to the above observed uniform convergence we conclude that $\langle T,\Phi^\#(\mathrm{d}\omega)\rangle=\langle\partial T,\Phi^\#\omega\rangle$ for any $T$ as above. However, thanks to the definition of distributional differential, the claimed identity is proved. 
\end{proof}

\begin{parag}[The tangential differential of a $1$-form along a surface]\label{tangentialdiff} Let $\Sigma$ be a $k$-dimensional graph of class $C^1$. Denote by $\mathrm{Tan}(\Sigma,x)$ the tangent of the surface $\Sigma$ at the point $x$ and let $e_1,\ldots,e_k$ be a family of continuous orthonormal vector fields that span $\mathrm{Tan}(\Sigma,x)$ at every $x\in \Sigma$. Denote with $e_1^*,\ldots,e_k^*$ the continuous $1$-forms for which $\scalar{e_i^*}{e_j}=\delta_{i,j}$ for $i,j\in \{1,\ldots,n\}$. For any Lipschitz $1$-form $\omega$ we define the tangential differential $d_\Sigma\omega$ of $\omega$ as
$$d_\Sigma\omega(x):=\sum_{i,j}\partial_{e_j}\omega_i(x)\,e_j^*\wedge dx_i,$$
where $\partial_{e_j}\omega_i(x)$ denotes the derivative of the Lipschitz function $\omega_i$ along the vector $e_j(x)$ at $x$. Such derivative exists and it is well defined for $\Haus^k\llcorner\Sigma$-almost every $x\in \R^n$ since $e_j(x)\in \mathrm{Tan}(\Sigma,x)$ for any $j$ and $x\in \Sigma$.
\end{parag}

\begin{parag}[Boundary and tangential differential]\label{bordoAM}
In the notations of \S \ref{tangentialdiff} let $X_1,\ldots,X_k$ be $k$ continuous vector fields on $\R^n$ be such that they span $\mathrm{Tan}(\Sigma,x)$ at every point $x$ of $\Sigma$.
Let $T:=X_1\wedge\ldots\wedge X_k\mathscr{H}^k\llcorner \Sigma$ and for every $x\in\Sigma$ and $r>0$ we define $T_{x,r}:=\mathbb{1}_{B(x,r)}T$. The currents $T_{x,r}$ are normal for every $x\in \Sigma$ and every $r>0$. Therefore, thanks to \cite[Proposition 5.13]{AlbMar}, we have
\begin{equation} 
    \begin{split}
        \scalar{\partial T_{x,r}}{\omega}=\int_{B(x,r)} \scalar{X_1\wedge\ldots\wedge X_k}{d_\Sigma \omega}\, d\mathscr{H}^k\llcorner \Sigma.
    \end{split}
\end{equation}    
\end{parag}

\begin{lemma}\label{prop:tangentialdifferential}
Let $\Phi:\Omega\to \R^n$ be a map of class $C^{1,1}$ where $\Omega$ is some open set of $\R^{k'}$, $\omega$ be a $1$-form of class $C^1$ on $\R^n$ and let $\Sigma$ be a $h$-dimensional surface of class $C^1$ in $\R^{k'}$. Then
$$\mathrm{d}_\Sigma(\Phi^\#\omega)(y)=\Phi^\#(\mathrm{d}\omega)(y)\vert_{\mathrm{Tan}(\Sigma,y)},$$
for every $y\in \Sigma$ such that $\Phi^\#\omega$ is differentiable along $\mathrm{Tan}(\Sigma,y)$ and a Lebesgue continuity point for $\mathrm{d}_\Sigma(\Phi^\#\omega)$ with respect to the measure $\mathscr{H}^h\llcorner \Sigma$. Note that the set of such $y$ is of full $\mathscr{H}^h\llcorner \Sigma$-measure. 
\end{lemma}

\begin{proof}
Let $X_1,\ldots,X_h$ be continuous vector fields in $\R^n$ such that 
$$\mathrm{span}(X_1(y),\ldots,X_h(y))=\mathrm{Tan}(\Sigma,y) \qquad\text{for $\mathscr{H}^h$-almost every $y\in \Sigma$.}$$
Let $B_1$ be the set of those $y\in \Sigma$ where $\Phi^\#\omega$ is differentiable along $\Tan(\Sigma,y)$, which is a set of full measure, since $\Phi^\#\omega$ is Lipschitz. Let us further notice that $\Phi^\#(d\omega)$ is a continuous $2$-form.
    For every $\alpha\in \Lambda^{h-2}(\R^n)$, $y\in \Sigma$ and $r>0$ let 
    $T_{y,r}^\alpha:= T_{y,r}\trace \alpha$, where the currents $T_{y,r}$ were introduced in \S\cref{bordoAM}. This implies in particular, thanks to Proposition \ref{id:distromega} we have
\begin{equation}
    \begin{split}
       &\scalar{T_{y,r}}{\Phi^\#(\mathrm{d}\omega)\wedge \alpha}=\scalar{T_{y,r}^\alpha}{\Phi^\#(\mathrm{d}\omega)}=\scalar{T_{y,r}^\alpha}{\mathrm{d}(\Phi^\#(\omega))}\\
       =&\scalar{T_{y,r}}{\mathrm{d}(\Phi^\#(\omega))\wedge \alpha}=\scalar{T_{y,r}}{\mathrm{d}(\Phi^\#(\omega)\wedge \alpha)}=\scalar{\partial T_{y,r}}{\Phi^\#(\omega)\wedge \alpha}.
       \nonumber
    \end{split}
\end{equation}
where the second last identity follows from the fact that $\alpha$ is constant and \cite[\S 2.8]{alberti2020geometric} and the last identity follows from Proposition \ref{lemmabordodistro}.
Furthermore, thanks to \S \ref{bordoAM} we infer that 
\begin{equation}
    \begin{split}
        \scalar{\partial T_{y,r}}{\Phi^\#\omega\wedge \alpha}=&\int_{B(y,r)} \scalar{X_1\wedge\ldots\wedge X_k}{d_\Sigma(\Phi^\#(\omega)\wedge \alpha)}\, d\mathscr{H}^k\llcorner \Sigma\\
        =&\int_{B(y,r)} \scalar{X_1\wedge\ldots\wedge X_k}{d_\Sigma(\Phi^\#\omega)\wedge \alpha}\, d\mathscr{H}^k\llcorner \Sigma.
    \end{split}
\end{equation}
We let $B_2$ be the subset of $B_1$ of those $w\in B_1$ that are Lebesgue continuity points of $\mathrm{d}_\Sigma(\Phi^\#\omega)$ with respect to $\mathscr{H}^h\llcorner \Sigma$.
Since $X_1\wedge\ldots\wedge X_k$ is a continuous $k$-vector, it is immediately seen that for 
every $y\in B_2$ we have 
$$\scalar{X_1(y)\wedge\ldots\wedge X_k(y)}{\Phi^\#(\mathrm{d}\omega)(y)\wedge \alpha}=\scalar{X_1(y)\wedge\ldots\wedge X_k(y)}{\mathrm d_\Sigma(\Phi^\#\omega)(y)\wedge \alpha}.$$
Since the choice of the vector fields $X_1,\ldots,X_k$ is arbitrary and since the above identity holds for every $y\in \Sigma$, we infer that 
$$\Phi^\#(\mathrm{d}\omega)(y)\wedge \alpha=\mathrm d_\Sigma(\Phi^\#\omega)(y)\wedge \alpha \quad\text{on $\Lambda_2(\mathrm{Tan}(\Sigma,y))$,}$$
for every $y\in B_2$ and every $\alpha\in \Lambda^{h-2}(\R^n)$. This concludes the proof of the lemma.
\end{proof}

\subsection{Rectifiability}

\begin{parag}[Rectifiability in metric spaces]\label{def:rect}
Given a metric space $(X,d)$ we say that a $\mathscr{H}^k$-measurable set $E\subseteq X$ is $k$-rectifiable if there are countably many compact set $K_i\subseteq \R^k$ and Lipschitz maps $f_i:K_i\to X$ such that 
$$\mathscr{H}^k_d\Big(E\setminus \bigcup_{i\in\N}f_i(K_i)\Big)=0.$$
Here as usual with $\mathscr{H}^k_d$ we denote is the $k$-dimensional Hausdorff measure constructed with the distance $d$ on $X$. See \cite[\S 2.10.2]{Federer1996GeometricTheory} for a reference. 
\end{parag}

\begin{parag}[Unrectifiability in metric spaces]
Let $(X,d)$ and $E$ be as in \S\ref{def:rect}. We say that $E$ is \emph{$k$-purely unrectifiable} if for every compact set $K\subseteq \R^n$ and any Lipschitz map $f:K\to X$, we have
$\mathscr{H}^k_d(f(K)\cap E)=0$.
\end{parag}

\subsection{Carnot-Carath\'eodory spaces}

Throughout this section we suppose $V$ to be a fixed distribution of $k$-planes of class $C^\infty$. 

\begin{parag}[Horizontal curves and Carnot-Carath\'eodory distance]\label{ccdistance}
An absolutely continuous curve $\gamma:[0,1]\to \R^n$ is said \emph{horizontal with respect to $V$} or simply \emph{horizontal} if, for $\Leb^1$-almost every $t\in [0,1]$, one has
$$\gamma^\prime(t)\in V(\gamma(t)).$$
For every $x,y\in\Omega$, we define \emph{Carnot-Carath\'eodory extended distance} as
\begin{equation} 
    d_V(x,y):=\inf\{\ell(\gamma):\gamma\text{ is a horizontal curve, }\gamma(0)=x\text{, }\gamma(1)=y\}.\nonumber
\end{equation}
where $\ell$ is the Euclidean length of $\gamma$.
\end{parag}

\begin{parag}[\cc spaces]\label{par:hor}
We say that $V$ satisfies the \emph{H\"ormander condition}, if at every point $x$ of $\R^n$, there exists a neighborhood $U$ of $x$, $N(x)\in\N$ and 
vector fields $X_1,\ldots, X_k$ spanning $V$ in $U$
such that the elementary commutators of $X_1,\ldots, X_k$ of length at most $N(x)$ span $\R^n$ at $x$. An important consequence of the H\"ormander condition is the following well known result. 

\begin{theorem}[{[Chow-Rashevskii, see \cite[p.95, \S 0.4]{Gromov1996Carnot-CaratheodoryWithin}]}]\label{CR}
Suppose there exists a system of $k$ vector fields
$\mathcal{X}$ spanning $V$ that satisfies the H\"ormander condition. Then, every couple of points $x,y$ in $\R^n$ can be joined by a horizontal curve. In particular the extended distance $d_V$ is actually a distance.
\end{theorem}

The metric spaces of type $(\R^n,d_V)$ are commonly known as \emph{\cc spaces}. Note that in all the above statements and definitions we could harmlessly replace $\R^n$  with an open and connected set. However, since we are interested essentially in local properties of our space for the sake of exposition we limit ourselves to the discussion of \cc spaces defined on $\R^n$.
\end{parag}

\begin{parag}[Exponential map and shape of \cc balls]\label{par:BBT}

Suppose $V$ satisfies the H\"ormander condition and let $\mathcal{X}=\{X_1,\ldots,X_k\}$ be a system of $k$ vector fields spanning $V$. 
Since $V$ satisfies the H\"ormander condition, this implies in particular that for every $x\in\R^n$ there is an open neighbourhood $U_x$ of $x$, which we refer to as the \emph{patch centered at }$x$, and $n-k$ vector fields $X_{k+1}\ldots, X_n$, which are simple commutators of $\mathcal{X}$ of length less than $N(x)$ 
such that $X_1(\zeta),\ldots,X_n(\zeta)$ are linearly independent at every $\zeta\in \text{cl}(U_x)$.

For every $\zeta\in U_x$ we let $\mathrm{Exp}_\zeta: U(0,\newr\label{r:4})\to\R^n$ be the function:
$$\mathrm{Exp}_\zeta(t_1,\ldots,t_n)=\Phi_{\sum_{j=1}^nt_jX_j}(\zeta,1),$$
where $\Phi_{\sum_{j=1}^nt_jX_{j}}(\zeta,1)$ is the flow of the vector field $\sum_{j=1}^nt_jX_j$ starting at $\zeta$ and computed at time $1$ and $\oldr{r:4}=\oldr{r:4}(x)$ is a sufficiently small positive constant for which $\Phi_{\sum_{j=1}^nt_jX_{j}}(\zeta,1)$ is well defined for every $\sum_{i}^n\lvert t_i\rvert\leq\oldr{r:4}$. 
The map $\mathrm{Exp}_\zeta$ is smooth, since the vector fields $X_1,\ldots, X_k$ are smooth, and hence for every $\zeta\in U_x$  and every $\lvert t\rvert\leq \oldr{r:4}$, there is $s\in[0,1]$ such that
\begin{equation}
\mathrm{Exp}_\zeta(t)=\mathrm{Exp}_\zeta(0)+\sum_{i=1}^n t_iX_i(\zeta)+\frac{\mathrm{d}^2\Exp_\zeta(st)[t,t]}{2}.
    \label{eq:10}
\end{equation}
By construction, the operator norm of $\mathrm{d}\Exp_{\zeta}(0)$ is (uniformly in $\zeta$) bounded away from $0$ on $U_x$, and thus thanks to the Inverse Function Theorem, there is $\oldr{nr1}=\newr\label{nr1}(x)>0$ such that for every $0<\rho<\oldr{nr1}$, there exists $\delta(\rho)>0$ for which
\begin{equation}
    U(\zeta,\delta(\rho))\subseteq \Exp_{\zeta}(U(0,\rho)),\qquad\text{whenever }\zeta\in U_x.
    \label{exp:open}
\end{equation}

The following theorem is commonly known as the Ball-Box theorem and it characterises the structure of small \cc balls. For a proof we refer for instance to \cite[Theorem 7]{nagel}. Here below we state the structure theorem for balls in a simpler way that is sufficients for our scopes.

\begin{theorem}[{\cite[p.98, \S 0.5.A]{Gromov1996Carnot-CaratheodoryWithin}}]\label{BallBox}
In the notations above, for every $x\in\Omega$ there are constants $\newC\label{C:1}=\oldC{C:1}(x)>1$ and $\newr\label{r:7}=\oldr{r:7}(x)>0$  such that:
$$
\Exp_\zeta(Q(\oldC{C:1}^{-1}\rho))\subseteq B_{d_V}(\zeta,\rho)\subseteq \Exp_\zeta(Q(\oldC{C:1}\rho)),\qquad\text{ for every }0<\rho<\oldr{r:7},$$
whenever $\zeta\in U_x$ and where $Q(\rho)$ is the \emph{anisotropic box} of side $\rho>0$, i.e.: 
$$Q(\rho):=\big\{t\in\R^n:\lvert t_j \rvert\leq \rho^{\text{deg}X_j}\text{ for every }j=1,\ldots,n\big\}.$$
\end{theorem}
\end{parag}

\subsection{Kirchheim's metric area formula}

In this subsection we briefly recall the notations and some few results necessary to state the area formula B. Kirchheim proved in \cite{Kirchheimarea}. In the following we assume that $f:\R^m\to(X,\lVert\cdot\rVert)$ is a Lipschitz map taking values in some Banach space $(X,\lVert \cdot\rVert)$. The following proposition, shows that the partial metric derivatives of $f$ along a fixed vector exist $\Leb^m$-almost everywhere.

\begin{proposition}[(Proposition 1, \cite{Kirchheimarea})]\label{prop:parder}
Let $f:\R^m\to(X,\lVert\cdot\rVert)$ be a Lipschitz map and suppose $u\in\mathbb{S}^{m-1}$. Then, for almost every $x\in\R^m$, $\lim_{r\to0}\lVert f(x+ru)-f(x)\rVert/r$ exists.
\end{proposition}

For every $x,u\in\R^m$ we let
$$MD(f,x)[u]:=\lim_{r\to 0}\frac{\lVert f(x+ru)-f(x)\rVert}{r},$$
whenever this limit exists. Thanks to Proposition \ref{prop:parder}, we know that $MD(f,x)[u]$ exists for any $u\in\mathbb{S}^{m-1}$ and  $\Leb^m$-almost every $x\in\R^n$, and in this case we say that $MD(f,x)$ is the \emph{metric differential} of $f$ at $x$. The following result is a metric version of the Rademacher's theorem.

\begin{theorem}[(Theorem 2, \cite{Kirchheimarea})]\label{metricrademacher}
Let $f:\R^m\to(X,\lVert \cdot\rVert)$ be a Lipschitz map. Then, for almost every $x\in\R^m$, $MD(f,x)[\cdot]$ is a seminorm on $\R^m$ and:
\begin{equation}
\lVert f(z)-f(y)\rVert-MD(f,x)[z-y]=o(\lvert z-x\rvert+\lvert y-x\rvert).
\label{eq:1031}
\end{equation}
\end{theorem}

\begin{theorem}\label{areaformula}
Let $f:\R^m\to (X,\lVert \cdot\rVert)$ be a Lipschtiz map, and $A\subseteq \R^m$ be a Lebesgue measurable set. Then
\begin{equation}
    \int \mathfrak{J}f(x)\Leb^m(x)=\int_X N(f\rvert_A,x)d\Haus^m_{\lVert \cdot\rVert}(x)
    \nonumber
\end{equation}
where $N(f\rvert_A,x)$ denotes the cardinality of the set $A\cap f^{-1}(x)$ and
$$\mathfrak{J}f(x):=m\Leb^m(B_1(0))\bigg(\int_{\mathbb{S}^{m-1}} (MD(f,x)[u])^{-m}d\Haus^{m-1}(u)\bigg)^{-1}.$$
\end{theorem}

Theorem \ref{areaformula} implies the following corollary.

\begin{corollary}\label{cor:rk}
Let $K$ be a compact subset of $\R^m$ and let $f:K\to (X,\lVert\cdot\rVert)$ be a Lipschitz map. Let $\mathcal{N}$ be the subset of those $x\in K$ at which $MD(f,x)$ exists, is a seminorm and there is
$u\in\mathbb{S}^{m-1}$ such that $MD(f,x)[u]=0$.
Then $\Haus^m_{\lVert\cdot\rVert}(f(\mathcal{N}))=0$.
\end{corollary}

\begin{proof}
Thanks to Theorem 2 in \cite{linbanach}, we can find a Lipschitz map $g:\R^m\to (X,\lVert\cdot\rVert)$ such that $g\lvert_K=f$. Moreover,  Theorem \ref{metricrademacher} implies that at density points $x$ of $K$ where $MD(f,x)$ is well defined, we also have $$MD(g,x)=MD(f,x).$$
In order to prove the proposition, thanks to Theorem \ref{areaformula}, it is sufficient to show that $\mathfrak{J}g=0$ $\Leb^m$-almost everywhere on $\mathcal{N}$. To show this, fix $x\in\mathcal{N}$  and assume that $u\in\mathbb{S}^{m-1}$ is such that $MD(f,x)[u]=0$. Since $MD(f,x)$ is a seminorm, we deduce that:
$$MD(f,x)[v]\leq MD(f,x)[u]+MD(f,x)[v-u]=MD(f,x)[v-u]\leq C\lvert v-u\rvert,$$
where $C:=\max_{e\in\mathbb{S}^{m-1}}MD(f,x)[e]$. This implies that $\mathfrak{J}(x)=0$ since:
$$\int_{\mathbb{S}^{m-1}} (MD(f,x)[v])^{-m}d\Haus^{m-1}(v)\geq C^{-m}\int_{\mathbb{S}^{m-1}} \lvert v-u\rvert^{-m}d\Haus^{m-1}(v)=\infty.$$
This concludes the proof.
\end{proof}

\subsection{Derived sets and the Mean Value Theorem}

Given a Lipschitz curve $\gamma:[0,1]\to \R^n$ for every $t\in[0,1]$ we define the \emph{derived set} of $f$ at $t$ as
\begin{equation}
    \mathcal{D} \gamma(t):=\bigcap_{r>0}\overline{\Big\{\frac{\gamma(a)-\gamma(b)}{a-b}:a<b\text{ and }\lvert a-t\rvert+\lvert b-t\rvert\leq r\Big\}}.
    \label{def:derived}
\end{equation}
If $\gamma$ is differentiable at $t$ the derived set clearly coincides with the singleton $\{\gamma'(t)\}$.

It is easy to see that for every $t\in[0,1]$ the set $\mathcal{D}\gamma(t)$ is compact and $\text{diam}_{eu}(\mathcal{D}\gamma(t))\leq \text{Lip}(\gamma)$.

\begin{proposition}[{[Mean Value Theorem]}]\label{prop:mean}
Let $\gamma\colon [0,1] \to \R^n$ be a Lipschitz curve. Then, for any $s,t \in [0,1]$, we have
\[
\lvert \gamma(s) - \gamma(t) \rvert \le \left( \sup_{t\in [0,1]} \max_{v\in \mathcal{D}\gamma(t)} \lvert v \rvert \right) \lvert s-t \rvert.
\]
Moreover, if $\gamma(0)=\gamma(1)$ then for every $w\in \R^n$ there exists some $t\in [0,1]$ such that
\[
0 \in \mathcal{D}\bigl(\langle \gamma, w \rangle\bigr)(t),
\qquad\text{where $\langle \gamma, w \rangle(t) := \langle \gamma(t), w \rangle$.}\]
\end{proposition}

\begin{proof}Since $\gamma$ is Lipschitz we have
$$\lvert \gamma(s)-\gamma(t)\rvert\leq \int_t^s \lvert \dot{\gamma}(\tau)\rvert d\tau\leq \sup_{t\in[0,1]}\max_{v\in\mathcal{D}\gamma(t)}\lvert v\rvert \lvert s-t\rvert.$$
Let us move to the second part of the proof. Assume $\gamma(0)=\gamma(1)$ and fix any $w\in \R^n$. Define the scalar function 
\[
f(t) := \langle \gamma(t), w\rangle.
\]
If $f$ is constant, there is nothing to prove. On the other hand, if is non-constant, continuous and $f(0)=f(1)$, it attains a maximum (or a minimum) at some interior point $t\in (0,1)$. Without loss of generality, suppose that $t\in (0,1)$ is a point where $f$ attains its maximum.
    For each $i\in \mathbb{N}$, define the level
    \[
    \ell_i = f(0) + \left(1-\frac{1}{i}\right)\bigl(f(t)-f(0)\bigr),
    \]
    which satisfies $\ell_i < f(t)$ but $\ell_i \to f(t)$ as $i\to\infty$. By continuity of $f$, for every $i$ there exist points $a_i,b_i\in (0,1)$ with $a_i < t < b_i$ such that
    $f(a_i) = f(b_i) = \ell_i$. For such pairs we clearly have 
    \[
    \frac{f(b_i)-f(a_i)}{b_i-a_i}  = 0.
    \]
    Now, either both sequences $\{a_i\}$ and $\{b_i\}$ converge to $t$ and thus by definition we have that $0\in \mathcal{D}f(t)$ or one of the sequences does not converge to $t$. In this case, the level $\ell_i$ is attained on an interval of nonzero length. This means that $f$ is locally constant on some open subinterval of $[0,1]$, and on such an interval every difference quotient is $0$. Consequently, for any point in that subinterval, $0\in \mathcal{D}f$. This concludes the proof.
\end{proof}

\section{An Euclidean unrectifiability result}
\label{sec:not}

This section is devoted to prove some results on the structure of tangency sets of smooth non-involutive distributions with $C^{1,1}$ and $C^2$ surfaces. 

\subsection{Statement of the results}

\begin{parag}[Tangency sets of a surface to a distribution]
\label{def:tan}
Let $k^\prime\leq k$ and suppose $V$ is a $k$-dimensional distribution of class $C^1$. 
Given a $k'$-dimensional manifold of 
class $C^1$ in $\R^n$ we say that $p\in S$
is a \emph{contact point} of $S$ to $V$ 
if and only if $\Tan(S,p)\subseteq V(x)$. 
The family of all such points is denoted by
$$\mathscr{C}(S,V):=\{q\in S:\Tan(S,q)\subseteq V(q)\}.$$
In the following we will say that a subset $E$ of $S$ is a \emph{tangency set} if $E\subseteq \mathscr{C}(S,V)$.
\end{parag}

The following theorem is an immediate consequence of Proposition \ref{th:unrectC2}. The rigidity given by the $C^2$ regularity yields the following surprisingly strong result. 

\begin{parag}[Theorem. ({Structure of tangency sets to $C^{2}$ surfaces})]\label{th:main1}
    Let $2\leq h\leq k'\leq k<n$ and suppose $V$ is an $h$-non-involutive $k$-dimensional distribution of class $C^1$ in $\R^n$. Let $S$ be a submanifold of class $C^{2}$ of $\R^n$ of dimension $k'$. Then $\mathscr{C}(S,V)$ is $(h-1)$-rectifiable.
\end{parag}

\begin{proof}
    For any $p\in \mathscr{C}(S,V)$ since $S$ is a submanifold and $V$ is of class $C^1$, there exists a neighbourhood $U$ of $p$ such that up to isometries, we can assume that $\mathrm{Tan}(S,p)=\mathrm{span}(\{e_1,\ldots,e_{k'}\})\cong \R^{k'}$ and that $S$ coincides in $U$ with a graph of a $C^{1,1}$ function $\varphi:\R^{k'}\to \R^{n-k'}$. This observation, together with Propositions \ref{th:unrectC2} concludes the proof. 
\end{proof}

If the regularity of the surface drops just to $C^{1,1}$, the result becomes the following natural one. Its proof is identical to that of Theorem \ref{th:main1} where instead of Proposition \ref{th:unrectC2} we employ Proposition \ref{prop:unrectC1,1}.

\begin{parag}[Theorem. ({Structure of tangency sets to $C^{1,1}$ surfaces})]\label{th:main2}
    Let $2\leq h\leq k'\leq k<n$ and suppose $V$ is an $h$-non-involutive $k$-dimensional distribution of class $C^1$ in $\R^n$ of dimension $k'$. Let $S$ be a submanifold of class $C^{1,1}$ of $\R^n$. Then $\mathscr{C}(S,V)$ is $h$-purely-unrectifiable.
\end{parag}

\subsection{Proofs}

\begin{parag}[Notation]\label{notationpureunrect} Let $2\leq h\leq k'\leq k< n$, $V$ be an $h$-non-involutive $k$-dimensional distribution of class $C^1$. Let $\Omega\subseteq \R^{k'}$ be an open set and suppose that $\varphi:\Omega\to \R^{n-k'}$ is a $C^{1,1}(\Omega,\R^{n-k'})$ map. Identified $\R^{k'}$ with $\mathrm{span}(e_1,\ldots,e_{k'})$ in $\R^n$, we denote by $\Phi:\Omega\subseteq \R^{k'}\to \R^{n}$ the map $\Phi(x)=(x,\varphi(x))$ and let 
$$E(\varphi,V):=\Phi^{-1}(\mathscr{C}(\mathrm{im}\Phi,V)).$$
Note that $\Phi$ is bi-Lipschitz on its image.
\end{parag}

\begin{parag}[Tangent of a Borel set]
    Let $E$ be a subset of $\R^n$. For every $z\in E$ we denote by $S(E,z)$ the set of those $v\in \mathbb{S}^{n-1}$ for which there exists a sequence of $z_i\in E$ such that $\lim_{i\to \infty}z_i=z$ and
    $$\lim_{i\to \infty}\frac{z_i-z}{\lvert z_i-z\rvert}=v.$$
    Finally, we let $\Tan(E,x):=\mathrm{span}(S(E,x))$.
\end{parag}

\begin{proposition}\label{th:unrectC2}
    Suppose $\varphi$ is of class $C^2$. Then $E(\varphi,V)$ is $(h-1)$-rectifiable.
\end{proposition}

\begin{proof}
    Throughout the proof of the proposition, let us fix $z\in E(\varphi,V)$ and let $\{z_i\}$ be a sequence in $E(\varphi,V)$ such that $T:=\mathrm{Tan}(E(\varphi,V),z)=\mathrm{Tan}(\{z_i:i\in\N\}\cup\{z\},z)$ and $\lim_{i\to\infty} z_i=z$ and let $v_j$ be an orthonormal basis of the $T$. 
Let us suppose that $t:=\mathrm{dim}(T)\geq h$.
For every $j=1,\ldots,t$ we denote by $\bar Y_j:\mathrm{im}\Phi\to \R^n$ the $C^1$ vector fields
$$\bar Y_j(p):=D\Phi(\Phi^{-1}(p))[v_j],\qquad \text{for any $p\in \mathrm{gr}_\Omega(\varphi)$.}$$
Denote $q_i:=\Phi(z_i)$ and $q:=\Phi(z)$ and let $B$ be the closed set $B:=\{q_i:i\in\N\}\cup \{q\}$.
It is immediate to see that there exists a neighbourhood $U$ of $q$ in $\R^n$ where $\bar Y_1,\ldots,\bar Y_t$ are independent in $U$.
Without loss of generality we can assume that there are $C^1$-vector fields $Z_1,\ldots,Z_k$ that span $V$ on $U$ and defined
$$W_j:=\sum_{\ell=1}^k\langle \bar Y_j,Z_\ell\rangle Z_\ell,\qquad \text{for any }j=1,\ldots,t,$$
we obtain a family of vector fields tangent to $V$ and such that $W_j(q)=\bar Y_j(q)$ for any $q\in B\cap U$. Without loss of generality, we can further assume that the vector fields $W_j$s are also independent on $U$, since they are independent at $q$. The vector fields $W_1,\ldots, W_t$ span a $t$-dimensional distribution $W$ of class $C^1$ tangent to $V$ in $U$. Therefore, since $V$ is supposed to be $h$-non-involutive, there exists a $1$-form $\omega$ of class $C^1$ such that $\omega\rvert_W=0$ but $\mathrm{d}\omega\rvert_W\neq 0$ in a neighbourhood of $q$, that we can suppose to coincide with $U$. It is easily seen that  $\Phi^\#\omega\vert_T=0$ and that $\mathrm{d}(\Phi^\#\omega\vert)\rvert_T\neq 0$ in a neighbourhood of $z$. However, since $\Phi^\#\omega\vert_T=0$ on $B$, we infer that $\mathrm D(\Phi^\#\omega)\rvert_T=0$ at $z$ and since $\mathrm d\Phi^\#\omega[X\wedge Y]=\mathrm D(\Phi^\#\omega)[X,Y]-\mathrm D(\Phi^\#\omega)[Y,X]$, we conclude that there must hold that $\mathrm d(\Phi^\#\omega)\rvert_T(z)=0$, yielding a contradiction. This implies in particular that $\mathrm{dim}(T)\leq h-1$. Thanks to \cite[Lemma 15.13]{Mattila1995GeometrySpaces}, we conclude that $E(\varphi,V)$ is $(h-1)$-rectifiable.  
\end{proof}

\begin{proposition}\label{prop:unrectC1,1}
If $\varphi$ is of class $C^{1,1}$, the set $E(\varphi,V)$ is $h$-purely unrectifiable. 
\end{proposition}

\begin{proof}
Let $\Gamma$ be an $h$-dimensional $C^1$ surface in $\R^k$ and suppose by contradiction that $$\mathscr{H}^h\llcorner \Gamma(E(\varphi,V))>0.$$
Let us recall that for $\mathscr{H}^h\llcorner \Gamma$ almost every $z\in E(\varphi,V)$ we have
\begin{itemize}
    \item[(i)] $\mathrm{Tan}(\Gamma,z)\subseteq \mathrm{Tan}(E(\varphi,V),z)$;
    \item[(ii)] $D\Phi$ is differentiable along $\mathrm{Tan}(\Gamma,z)$ at $z$.
\item[(iii)]$z$ is a Lebesgue continuity point for $\mathrm{d}_\Gamma(\varphi^\#\omega)$ with respect to the measure $\mathscr{H}^h\llcorner \Gamma$.
\end{itemize}
Let us fix such a $z$ and let $T:=\mathrm{Tan}(\Gamma,z)$. 
Let $\{z_i\}$ be a sequence in $\Gamma\cap E(\varphi,V)$ such that $T=\mathrm{Tan}(\{z_i:i\in\N\}\cup\{z\},z)$ and $\lim_{i\to\infty} z_i=z$. Let $v_j$ be an orthonormal basis of $T$ notice that since $\Gamma$ is an $h$-Lipschitz graph, we have that $\mathrm{dim}(T)\leq h$. Finally, we assume that at $z$ the function $D\Phi$ is differentiable along $T$. 
For any $j=1,\ldots,h$ we denote by $Y_j:\mathrm{im}\Phi\to \R^n$ the Lipschitz vector fields
$$Y_j(p):=D\Phi(\Phi^{-1}(p))[v_j],\qquad \text{for any $p\in \mathrm{im}\Phi$.}$$
Denote $q_i:=\Phi(z_i)$ and $q:=\Phi(z)$ and let $B$ be the closed set $B:=\{q_i:i\in\N\}\cup \{q\}$, we now check that there are $C^1$ vector fields $\bar{Y}_j$ defined on $\R^n$ that coincide with $Y_j$ on $B$. Since the only accumulation point of $B$ is by construction $q$, we infer that the existence of the vector fields $\bar{Y}_j$ is ensured by checking that the hypothesis of Whitney's theorem, see \cite[3.1.14]{Federer1996GeometricTheory}, at $q$. However, this is immediately seen since we are assuming that $D\varphi$ is differentiable along $v_j$ at $z$ for any $j$. Since $\bar Y_j$ coincides with $Y_j$ on $B$, we infer that in a neighbourhood of $q$ they are linearly independent and 
$$\mathrm{span}(\{\bar Y_1(p),\ldots,\bar Y_h(p)\})=\mathrm{Tan}(\mathrm{im}\Phi,p)\subseteq V(p)\qquad\text{for any }p\in B.$$
It is immediate to see that there exists a neighbourhood $U$ of $q$ in $\R^n$ where $\bar Y_1,\ldots,\bar Y_h$ are independent in $U$. With the same argument employed in the proof of Theorem \ref{th:unrectC2}, we can construct vector fields $W_1,\ldots,W_h$ of class $C^1$ that are independent and tangent to $V$ on $U$. In addition, $W_j=Y_j$ on $B\cap U$.

Let $W$ be the $h$-dimensional distribution of class $C^1$ spanned by $W_1,\ldots,W_h$ and note that since $V$ is $h$-non-involutive, there exists a $1$-form of class $C^1$ such that such that $\omega\rvert_W=0$ and $\mathrm{d}\omega\rvert_W\neq 0$. 
It is immediate to see that under the above the hypothesis, we have $\Phi^\#\omega\rvert_T=0$ on $\Phi^{-1}(U)$. In addition, it is easy to check that $\Phi_\#(\mathrm{d}\omega)\lvert_T\neq 0$ on $\Phi^{-1}(U)$. However, thanks to our assumptions on $z$, we infer that Lemma \ref{prop:tangentialdifferential} can be applied at $z$ to yield  
$$\mathrm{d}_\Gamma(\Phi^\#\omega)(z)=\Phi^\#(\mathrm{d}\omega)(z)\vert_{T}\neq 0.$$
However, it is immediately seen that since $\Phi^\#\omega\rvert_T(z)=0$, thanks to its very definition we also have that $\mathrm{d}_\Gamma(\Phi^\#\omega)(z)=0$. This results in a contradiction and therefore $\mathscr{H}^h(\Gamma\cap E(\varphi,V))=0$, concluding the proof of the proposition.  
\end{proof}

\section{Pure unrectifiability of Carnot-Carath\'eodory spaces}

\subsection{Main result and strategy of the proof}

Suppose $V$ is a smooth distribution of $k$-planes with the H\"ormander condition. Let $1<h\leq k$ be the smallest positive integer for which $V$ is $h$-non-involutive. The following theorem is the main result of this section. 

\begin{theorem}\label{TH:UN}
Suppose $K$ is a compact subset of $\R^m$ for some $m\geq h$ and $f$ is a Lipschitz map from $K$, endowed with the Euclidean distance, to the \cc space $(\R^n,d_V)$. Then
 $$\Haus^{m}_{d_V}(f(K))=0.$$
\end{theorem}

Theorem \ref{TH:UN} implies that the metric space $(\R^n,d_V)$ is $(\Haus^m_{d_V},m)$-purely unrectifiable for any $m\geq h$, see  \cite[\S 3.2.14]{Federer1996GeometricTheory} for a definition. This fact is well known in the case of Carnot groups, see for instance \cite{AK00} and \cite{HAL}, but to our knowledge it was not still proved for general \cc spaces.   

The strategy of the proof of Theorem \ref{TH:UN} is to first show that one can reduce to the case in which $f$ is bi-Lipschitz, see Lemma \ref{lemma:AK}, and then prove that, for any point $x\in K$ at which $f$ is differentiable, we have 
\begin{equation}
    \text{im}(\mathrm{d}f(x))\subseteq V(f(x)).
    \label{eq:25}
\end{equation}
The above inclusion, thanks to \cite[Theorem 5.3]{HAL} directly implies that if $m>k$, the set $f(K)$ is $\Haus^m_V$-null.

In order to prove the claim in the case $h\leq m \leq k$, we first establish that $f(K)$ exhibits a quadratic tangency structure at almost every point, described precisely by the inclusion \eqref{incl:cono2}. This geometric regularity, combined with the inclusion \eqref{eq:25} and Proposition \ref{prop:unrectC1,1}, ensures that $f(K)$ is $\Haus^m$-null.

Now, consider a countable covering $\{U(x_l, \rho_l)\}_{l \in \N}$ of $f(K)$ with Euclidean balls centered at points in $f(K)$, satisfying $\sum_{l \in \N} \rho_l^m \leq \epsilon$. From this covering, we construct a refined covering by intersecting each $U(x_l, \rho_l)$ with the cone $\Exp_{x_l}(X(\lambda))$, as specified in inclusion \eqref{incl:cono}. This refined family of Borel sets remains a covering of $f(K)$, and each element's diameter, measured with respect to the distance $d_V$, is still comparable to $\rho_l$.
 
\subsection{Regularity of the images of Lipschitz functions}

Throughout this subsection we let $C$ be a compact subset of  $\R^m$ and we let $g:C\to (\R^n,d_V)$ be a fixed Lipschitz map. The following proposition is a special case of  \cite[Lemma 4.1]{metriccurrents}. It allows us to reduce to study the case in which $g$ is a bi-Lipschitz map.

\begin{lemma}\label{lemma:AK}
There are countably many compact sets $K_i\subseteq\R^m$  and bi-Lipschitz maps $f_i:K_i\to(\R^n,d_V)$ such that
$$\Haus^m_{d_V}\bigg(g(C)\setminus \bigcup_{i\in\N} f_i(K_i)\bigg)=0.$$
\end{lemma}

The following proposition is a direct consequence of the shape of \cc balls, i.e. of Theorem \ref{BallBox}, and it will be used to prove that Lipschitz images of Euclidean spaces inside $(\R^n,d_V)$ must be very regular.

\begin{proposition}\label{prop:coni}
Suppose $K$ is a compact subset of $\R^m$,  
$f:K\to (\R^n,d_V)$ is an $L$-bi-Lipschitz map with constant $L>1$ for which there are
$\lambda,\overline{r}>0$ such that
$$\lambda\lvert z-y\rvert\leq \lvert f(z)-f(y)\rvert,\qquad\text{for every $z,y\in K$ provided $\lvert z-y\rvert\leq \overline{r}$.}$$
Then, for every $w\in K$ there are constants $\newr\label{r:8}=\oldr{r:8}(w,\overline{r})>0$ and $\Lambda=\Lambda(w,\lambda)>0$ such that
\begin{equation}
    f(K)\cap U(f(w),\oldr{r:8})\subseteq \Exp_{f(w)}\big(X(\Lambda)\big),
    \label{incl:cono}
\end{equation}
and where $X(\Lambda)$ is the \emph{anisotropic cone} of amplitude $\Lambda$, i.e.
$$X(\Lambda):=\{t\in U(0,\oldr{r:4}):\lvert t_l\rvert\leq (\Lambda \lvert t\rvert)^{\text{deg}X_l} \text{ for every }l=1,\ldots,n\}.$$
A useful consequence of \eqref{incl:cono} is the fact that 
\begin{equation}
\begin{split}
     &f(K)\cap U(f(w),\oldr{r:8})\\
     &\qquad\qquad\subseteq \Big\{x\in U(f(w),\oldr{r:8}) :\lvert \Pi_{V(f(w))^\perp}(x-f(w))\rvert\leq 2\Lambda^2\lvert x-f(w)\rvert^2\Big\},
    \label{incl:cono2}
\end{split}
\end{equation}
for every $w\in K$, 
where $\Pi_{V(f(w))^\perp}$ denotes the orthogonal projection onto $V(f(w))^\perp$.
\end{proposition}

\begin{proof}
Define $\newr\label{r:9}:=\min\{\oldr{r:4}/2\sqrt{n}\oldC{C:1}L,\oldr{nr1},\oldr{r:7}/L,\overline{r},1\}$, where $\oldr{r:4},\oldr{nr1},\oldr{r:7}$ and $\oldC{C:1}$ are the constants relative to the point $f(w)$ let $z\in K$ be such that $f(z)\in U_{f(w)}$. Thanks to Theorem \ref{BallBox} and the Lipschitzianity of $f$, for every $0<\rho<\oldr{r:9}$ we have that
\begin{equation}
    f(U(w,\rho)\cap K)\subseteq B_{d_V}(f(w),L\rho)\subseteq \Exp_{f(w)}\big(Q(\mathfrak{c}\rho)\big),
    \label{eq:30}
\end{equation}
where $\mathfrak{c}:=\oldC{C:1}L$. Thanks to the inclusion \eqref{eq:30}, for every $y\in U(w,\rho)\setminus U(w,\rho/2)\cap K$ there exists $\tau\in Q(\mathfrak{c}\rho)$ such that $f(y)=\Exp_{f(w)}(\tau)$. Therefore for every $l=1,\ldots,n$ we have
\begin{equation}
    \begin{split}
        \lvert \tau_l\rvert\leq& \big(\mathfrak{c}\rho\big)^{\text{deg}X_l}\leq \big(2\mathfrak{c}\lvert w-y\rvert\big)^{\text{deg}X_l}
        \leq \big(2\mathfrak{c}\lambda\lvert f(w)-f(y)\rvert\big)^{\text{deg}X_l},
        \label{eq:20}
    \end{split}
\end{equation}
where the second inequality above comes from the fact that $\rho\leq2\lvert z-y\rvert$ and the last one from the local Euclidean $\lambda$-Lipschitzianity of the inverse of $f$.
Moreover, thanks to our choice of $\oldr{r:9}$, we deduce that $\tau\in Q(\mathfrak{c}\oldr{r:9})\Subset U(0,\oldr{r:4})$ and thus identity \eqref{eq:10} implies that
\begin{equation}
    \begin{split}
    \lvert f(y)-f(w)\rvert=&\lvert\Exp_{f(y)}(\tau)-\Exp_{f(w)}(0)\rvert\\
&\leq \bigg(\sum_{l=1}^n \big\lvert X_l(f(w))\big\rvert+\sup_{t\in Q(\mathfrak{c}\oldr{r:9})}\lVert\mathrm{d}^2\Exp_{f(w)}(\tau)\rVert\bigg) \lvert \tau\rvert=:\mathfrak{m}\lvert \tau\rvert,
\label{eq:21}
\end{split}
\end{equation}
Summing up, inequalities \eqref{eq:20} and \eqref{eq:21} imply that $\lvert \tau_l\rvert\leq (2\lambda\mathfrak{c}\mathfrak{m}\lvert \tau\rvert)^{\text{deg}X_l}$ and thus
\begin{equation}
    f(U(z,\rho)\setminus U(z,\rho/2)\cap K)\subseteq \Exp_{f(z)}(X(2\lambda\mathfrak{c}\mathfrak{m})),
    \label{eq:22}
\end{equation}
for every $0<\rho<\oldr{r:9}$. The arbitrariness of $\rho$ in the inclusion \eqref{eq:22}, the $L$-bi-Lipschitzianity of $f$ and Theorem \ref{BallBox}, imply
\begin{equation}
    \begin{split}
f(K)\cap\Exp_{f(w)}\big(Q(\oldr{r:9}/\mathfrak{c})\big)\subseteq& f(K)\cap B_{d_V}(f(w),\oldr{r:9}/L)
        \\\subseteq& f(U(w,\oldr{r:9})\cap K)\subseteq \Exp_{f(w)}\big(X(2\lambda\mathfrak{c}\mathfrak{m})\big).
        \nonumber
    \end{split}
\end{equation}
Notice that in the second inclusion above we are using the fact that $f$ is supposed to be injective. 
Furthermore, since the box $Q(\oldr{r:9}/\mathfrak{c})$ contains the Euclidean ball $U(0,\newr\label{r:11})$, where $\oldr{r:11}:=(\oldr{r:9}/\mathfrak{c})^{N(f(w))}$, thanks to \eqref{exp:open} we finally conclude that:
$$f(K)\cap U\big(f(z),\delta(\oldr{r:11})\big)\subseteq \Exp_{f(z)}\big(X(2\lambda\mathfrak{c}\mathfrak{m})\big).$$
With the choice $\oldr{r:8}:=\delta(\oldr{r:11})$ and $\Lambda:=2\lambda\mathfrak{c}\mathfrak{m}$ the proof of the inclusion \eqref{incl:cono2} is complete. Note moreover that by construction, $\oldr{r:8}$ depends only on $w$ and $\overline{r}$ while
and $\Lambda$ depend only on $w$ and $\lambda$.

For any unitary vector $u\in V(f(w))^\perp$, by construction we have that $\langle u,X_l(f(w))\rangle=0$ for every $1\leq l\leq k$. Furthermore, since $\text{deg}X_l\geq 2$ whenever $l\geq k+1$, 
thanks to identities \eqref{eq:10}, \eqref{eq:20} and few algebraic computations which we omit, we deduce that:
$$\lvert\langle u, f(y)-f(w)\rangle\rvert\leq \bigg\lvert \sum_{l=k+1}^n\tau_l \langle u,X_l(f(w))\rangle\bigg\rvert+\mathfrak{m}\lvert \tau\rvert^2\leq 2\Lambda^2 \lvert f(w)-f(y)\rvert^2.$$
This proves the inclusion \eqref{incl:cono2} and thus proof of the lemma is complete.
\end{proof}

\begin{lemma}\label{lem:palle}
For every $\Lambda>0$ and every $x\in\R^n$ there are $\newC\label{C:2}=\oldC{C:2}(x,\Lambda)>0$ and $\newr\label{r:12}=\oldr{r:12}(x)>0$ such that for every $\zeta\in U_x$ and every $\rho<\oldr{r:12}(x)$, we have
$$U(\zeta,\rho)\cap \Exp_\zeta(X(\Lambda))\subseteq B_{d_V}(\zeta,\oldC{C:2}\rho).$$
\end{lemma}

\begin{proof}
For every $\rho\leq\min\{\oldr{r:4},\oldr{nr1},\oldr{r:7}\}=:\newr\label{r:5}$ and every $\omega\in U(\zeta,\rho)\cap \Exp_\zeta(X(\Lambda))$ there exists $t\in X(\Lambda)$ such that $\omega=\Exp_\zeta(t)$. Therefore, identity \eqref{eq:10} implies that:
\begin{equation}
   \begin{split}
          \rho\geq&\lvert w-\zeta\rvert=\big\lvert \Exp_\zeta(t)-\Exp_\zeta(0)\big\rvert=\big\lvert \mathrm{d}\Exp_\zeta(0)[t]+\mathrm{d}^2 \Exp_\zeta (st)[t,t]\big\rvert\\
          \geq& \min_{\substack{v\in \mathbb{S}^{n-1}\\\zeta\in U_x}} \big\lvert \mathrm{d}\Exp_\zeta(0)[v]\big\rvert \: \lvert t\rvert-\max_{\substack{\tau\in U(0,\oldr{r:4})\\\zeta\in U_x}} \big\lVert \mathrm{d}^2\Exp_\zeta(\tau)\big\rVert \:\lvert t\rvert^2=:\mathfrak{n}_1\lvert t\rvert-\mathfrak{n}_2\lvert t\rvert^2.
          \label{eq:24}
   \end{split}
\end{equation}
Recall that since $\mathrm{Exp}$ is a diffeomorphism in $U_x$ we have that $\mathfrak{n}_1>0$. 
If $\rho\leq \min\{\mathfrak{n}_1^2/8(\mathfrak{n}_2+1), \oldr{r:5}\}=:\oldr{r:12}$, inequality \eqref{eq:24} and some algebraic computations that we omit, imply that $\lvert t\rvert\leq 2\rho/\mathfrak{n}_1$.

This means that
$$\lvert t_i\rvert\leq (\Lambda \lvert t\rvert)^{\mathrm{deg}(X_i)}\leq \Big(\frac{2\Lambda}{\mathfrak{n}_1}\rho\Big)^{\mathrm{deg}(X_i)}\qquad\text{for every }i=1,\ldots,n,$$
and hence $t\in Q(2\Lambda \rho/\mathrm{n}_1)$.

In particular, thanks to Theorem \ref{BallBox} and the choice $\oldC{C:2}:=2\Lambda/\mathfrak{n}_1\oldC{C:1}$, we have
$$U(\zeta,\rho)\cap\Exp_{\zeta}(X(\Lambda))\subseteq \Exp_\zeta(Q(X(\oldC{C:2}/\oldC{C:1}\rho)))\subseteq B_{d_V}(\zeta,\oldC{C:2}\rho),$$
for every $0<\rho<\oldr{r:12}$. Finally, we remark that the constants $\oldr{r:12}$ and $\oldC{C:2}$ depend only on $x$ and $\Lambda$ since $\mathfrak{n}_1,\mathfrak{n}_2,\oldr{r:5}$ and $\oldC{C:2}$ depend only on $x$.
\end{proof}

\begin{proposition}\label{whitney}
    Suppose $K$ is a compact subset of $\R^m$,  
$f:K\to (\R^n,d_V)$ is an $L$-bi-Lipschitz map with constant $L>1$ for which there are
$\lambda,\overline{r}>0$ such that
$$\lambda\lvert z-y\rvert\leq \lvert f(z)-f(y)\rvert\qquad\text{for every $z,y\in K$ provided $\lvert z-y\rvert\leq \overline{r}$.}$$
For every $z\in f(K)$ there exists an $r_z>0$ such that the following holds. There exists a map $\varphi:V(z)\to \R^n$ of class $C^{1,1}$ such that 
$$\varphi(\Pi_{V(z)}(f(K)\cap B(z,r_z)))=f(K)\cap B(z,r_z),$$
where $\Pi_{V(z)}$ denotes the orthogonal projection onto $V(z)$ and
$$\mathrm{im}(d\varphi(z))\subseteq V(\varphi(z)) \qquad\text{for every }z\in \Pi_{V(z)}(f(K)\cap B(z,r_z)).$$
\end{proposition}

\begin{proof}
    Since $V(z)$ is continuous, there exists a neighbourhood $U$ of $z$ such that for every $w\in U$ we have that $V(w)$ is a graph of a linear function $M_w:V(z)\to V(z)^\perp$ over $V(z)$ and we can also assume without loss of generality 
    \begin{equation}
        \mathrm{dist}(V(w_1)\cap \mathbb{S}^{n-1},V(w_2)\cap \mathbb{S}^{n-1})\leq \frac{1}{10}\qquad\text{for every }w_1,w_2\in U.
        \label{vicipiani}
    \end{equation}
    Thanks to Proposition \ref{prop:coni}, there exists a possibly smaller neighbourhood $U'$ of $z$ such that 
    \begin{equation}
        \lvert \Pi_{V(f(w))^\perp}(f(a)-f(w))\rvert\leq 2\Lambda^2\lvert f(a)-f(w)\rvert^2
        \leq 2\Lambda^2L^2\lvert a-w\rvert^2,
        \label{stimeflungoV}
    \end{equation}
    whenever $a\in f^{-1}(U')$, where $\Lambda$ is the constant yielded by Proposition \ref{prop:coni} and where the last inequality comes from the fact that by construction the Euclidean distance is smaller than $d_V$. Proposition \ref{prop:coni} also implies that if $U'$ is chosen sufficiently small, 
    $\overline{U'}\cap f(K)$ is the graph of a Lipschitz function $\phi:F\to V(z)^\perp$, where $F:=\Pi_{V(z)}(\overline{U'}\cap f(K))$.

    Let us construct the map $\varphi$. As a first step, let us check that for every $x,y\in F$ we have 
    $$\phi(y)=\phi(x)+M_x(y-x)+R(y,x),$$
    with $\lvert R(y,x)\rvert\leq C\lvert y-x\rvert^2$ for some constant $C>0$. 
   So, let $x,y\in F$ and assume up to translations that $x=0$. Note that 
   \begin{equation}
       \begin{split}
           &\qquad\lvert\phi(y)-M_0(y)\rvert^2=\lvert(y+\phi(y))-y-M_0(y)\rvert^2\\
           =&\lvert \Pi_{V(0)^\perp}(y+\phi(y))\rvert^2+\lvert \Pi_{V(0)}((y+\phi(y))-y-M_0(y))\rvert^2\\
        \overset{\eqref{stimeflungoV}}{\leq}& 4\Lambda^4(1+\mathrm{Lip}(\phi))^2\lvert y\rvert^4+ \lvert \Pi_{V(0)}((y+\phi(y))-y-M_0(y))\rvert^2,
           \nonumber
       \end{split}
   \end{equation}
In addition, since by construction $y+M_0(y)\in V(0)$, we conclude that we are left to estimate $\lvert \Pi_{V(0)}(y+\phi(y))-y-M_0(y)\rvert$. Consider the triangle $T$ with vertices $y+M_0(y)$, $y+\phi(y)$ and $\Pi_{V(0)}(y+\phi(y))$. Let us call $\alpha$ the angle insisting at $y+\phi(y)$. Then, since the $T$ is right angled at $\Pi_{V(0)}(y+\phi(y))$, one can prove with some omitted algebraic computations that if $\Lambda \lvert y\rvert\leq 1$ and since $\tan(\alpha)\leq 1/10$, then
\begin{equation}
    \lvert \Pi_{V(0)}(y+\phi(y))-y-M_0(y)\rvert\leq 2\Lambda \lvert y\rvert^2\qquad\text{ whenever }\lvert y\rvert\leq \Lambda^{-1}.
\end{equation}
This finally implies that 
$$\lvert\phi(y)-M_0(y)\rvert^2\leq 4\Lambda^4(2+\mathrm{Lip}(\phi))^2\lvert y\rvert^4.$$
This shows in particular that the hypothesis of Whitney's extension theorem, see \cite[\S 2.3, Theorem 4]{zbMATH03441026}, are satisfied. Hence we can find a $C^{1,1}$ map $\varphi:V(z)\to V(z)^\perp$ such that
$$\varphi(w)=\phi(w)\qquad\text{and}\qquad D\varphi(w)=M_w\qquad\text{for every }w\in F.$$
This concludes the proof.
\end{proof}

\begin{proof}[Proof of Theorem \ref{TH:UN}]
 Thanks to Lemma \ref{lemma:AK} there are countably many $L_i$-bi-Lipschitz maps $f_i:K_i\Subset \R^m\to(\R^n,d)$ such that
$$\Haus^m_{d_V}\bigg(f(K)\setminus\bigcup_{i\in\N} f_i(K_i)\bigg)=0.$$
For any $\lambda,\overline{r}\in\Q^+$ and $i\in\N$, we let:
\begin{equation}
K_i(\lambda,\overline{r}):=\bigg\{z\in K_i:\lambda\leq \frac{\lvert f_i(y)-f_i(z)\rvert}{\lvert y-z\rvert}\text{ for any }y\in U(z,\overline{r})\cap K_i\bigg\}.
\label{eq:1027}
\end{equation}
Recall that the functions $f_i$ are locally Euclidean Lipschitz and thus their differential exists outside a null set $\mathcal{N}$. This implies that
\begin{equation}
K_i\setminus\bigcup_{\lambda,\overline{r}\in\Q^+} K_i(\lambda,\overline{r})\subseteq\{z\in K_i:\text{rk}(\mathrm{d}f_i(z))<m\}\cup \mathcal{N},
\label{eq:1023}
\end{equation}
and in particular
$$\Haus_V^m\bigg(f_i(K_i)\setminus \bigcup_{\lambda,\overline{r}\in\Q^+}f_i(K_i(\lambda,\overline{r}))\bigg)\leq \Haus_V^m\big(f_i(\{z\in K_i:\text{rk}(\mathrm{d}f(z))<m\})\big)=0,$$
where the last equality above is an immediate consequence of \cite[Theorem 5.3]{HAL}.

The above discussion shows that without loss of generality we can reduce to show that if $f$ is bi-Lipschitz and such that $K=K_i(\lambda,\overline{r})$ for some $\lambda,\overline{r}>0$, then $f(K)$ is $(\mathscr{H}^m_V,m)$-purely unrectifiable.  

Let us note that thanks to \eqref{incl:cono2}, it is immediate to see that 
$$\mathrm{im}(\mathrm{d}f(z))\subseteq V(f(z))\qquad\text{for every }z\in U(x,\overline{r})\cap K.$$
Thanks to Proposition \ref{whitney}, we know that there are countably many $C^{1,1}$ surfaces $S_i$ of dimension $k$ such that 
$$\{z\in U(x,\overline{r})\cap K:\mathrm{im}(\mathrm{d}f(z))\subseteq V(f(z))\}\subseteq \bigcup_{i\in\N} \mathscr{C}(S_i,V).$$
However, by Proposition \ref{prop:unrectC1,1} we know that $\mathscr{C}(S_i,V)$ is (Euclidean) $h$-purely unrectifiable for every $i\in \N$, i.e.
$$\Haus^h(\Gamma\cap f(K))=0,\qquad\text{for every $h$-rectifiable }\Gamma.$$

Now, assume for contradiction that there exists an \( m \)-dimensional Lipschitz surface \( S \) with
\[
\Haus^m(S \cap f(K)) > 0.
\]
Without loss of generality, we may suppose that \( S \) is compact and can be expressed as the graph of a Lipschitz function $\Psi: K \Subset W \to W^\perp$,
where \( W \) is an \( m \)-dimensional plane in \( \mathbb{R}^n \).

Since \( S \) has a positive \( m \)-dimensional measure intersecting \( f(K) \), we can apply a slicing argument. Let \( A \) be an \( (m-h) \)-dimensional subspace of \( W \). By Fubini’s theorem, the measure \( \Haus^m \) restricted to \( S \cap f(K) \) can be decomposed as an integral of \( h \)-dimensional measures:
\[
\Haus^m\llcorner (S \cap f(K)) = \int \Haus^h\llcorner (S_z \cap f(K)) \, d\Haus^{m-h}\llcorner A(z),
\]
where the slices \( S_z \) are defined by
\[
S_z := \Psi\bigl(\pi_W(S) \cap (z+B)\bigr),
\]
with \( B \) being the orthogonal complement of \( A \) in \( W \) and \( \pi_W \) the orthogonal projection onto \( W \).
This decomposition implies that there must exist a set of \( z \in A \) of positive \( \Haus^{m-h} \)-measure for which
\[
\Haus^h(S_z \cap f(K)) > 0.
\]
However, each slice \( S_z \) is an \( h \)-rectifiable set, and by the earlier conclusion its intersection with \( f(K) \) must have zero \( \Haus^h \)-measure. This contradiction shows that no such \( m \)-dimensional Lipschitz surface \( S \) can exist.

The final step is to prove that we can improve the above identity to 
$$\Haus^m_V(\Gamma\cap f(K))=0,\qquad\text{for every $m$-rectifiable }\Gamma.$$
Fix a compact $m$-rectifiable set $\Gamma$.
By compactness there are finitely many patches $U_{x_p}$, see \S\ref{par:BBT}, centred at points $x_p\in f(K)\cap \Gamma$ such that
$f(K)\cap  \Gamma\subseteq \bigcup_{p=1}^M U_{x_p}$. 
Since for every $p\in\{1,\ldots,M\}$ the set $f(K)\cap \Gamma\cap  U_{x_p}$ is $\Haus^m$-null, for every $\delta>0$ and $\epsilon>0$ there is a countable Borel cover $\{B_l\}_{l\in\N}$ of $f(K)\cap U_{x_p}$ such that defined $\rho_l:=\text{diam}(B_l)\leq \min\{\oldr{r:12}(x_p)/4,\delta\}$ for every $l\in\N$ we have
$$\sum_{l\in\N}\rho_l^m\leq \epsilon.$$

Let $z_l$ be an element of $f(K)\cap \Gamma\cap U_{x_p}\cap B_l$ and note that $\{U(z_l,2\rho_l)\}_{l\in\N}$ also covers $f(K)\cap U_{x_p}$.
Furthermore, since $2\rho_l<\oldr{r:12}(x_p)$, Corollary \ref{prop:coni} and Lemma \ref{lem:palle} imply that for every $l\in\N$ we have
\begin{equation}
U(z_l,2\rho_l)\cap f(K)\subseteq U(z_l,2\rho_l)\cap\Exp_{z_l}\big(X(\Lambda(x_p))\big)\subseteq B_{d_V}(z_l,4 \oldC{C:2}(x_p)\rho_l).
\nonumber
\end{equation}
Therefore, $\{B_{d_V}(z_l,4 \oldC{C:2}\rho_l)\}_{l\in\N}$ is a countable cover for $U_{x_p}\cap f(K)$ and
$$\sum_{l\in\N}\text{diam}_V\big(B_{d_V}(z_l,4 \oldC{C:2}(x_p)\rho_l)\big)^m\leq 4^m\oldC{C:2}(x_p)^m\sum_{l\in\N}\rho_l^m\leq 4^m\oldC{C:2}(x_p)^m\epsilon.$$
Thanks to the arbitrariness of $\delta>0$, we deduce that
$$\Haus^m_V(f(K))\leq\sum_{i=1}^M\Haus^m_V(f(K)\cap U_{x_p}) \leq 4^m\max_{p=1,\ldots,M}\oldC{C:2}(x_p)^mM\epsilon.$$
Thanks to the arbitrariness of $\epsilon$, we conclude that $\Haus_V^m(f(K))=0$.
\end{proof}


\section{Extremality of \cc spaces}

\subsection{Squeezed metrics}

Suppose $V$ and $W$ are distributions of $k$ and $(n-k)$-planes in $\R^n$, spanned by the systems vector fields $\mathcal{X}:=\{X_1,\ldots,X_k\}$ and $\mathcal{Y}:=\{X_{k+1},\ldots,X_n\}$ respectively, and assume that for every $x\in\R^n$ we have
$$\R^n=V(x)\oplus W(x).$$
Throughout this section, for every $x\in\R^n$, we denote with $\Pi_x^V$ the projection on $V(x)$ along the vector subspace $W(x)$ and with $\Pi_x^W$ the map $\text{id}-\Pi_x^V$, which is the projection on $W(x)$ along $V(x)$.

\begin{parag}[Squeezed metrics on $V$]
The idea behind the definitions and the computations of this section is that the rectifiability properties of \cc spaces is completely determined by the Ball-Box theorem. With this in mind, we give the following definition.

\begin{definition}\label{def:squmet}
For every $\eta\in[1,2]$ we say that a metric $d$ on $\R^n$ is $\eta$-squeezed on $V$ if every compact set $K\subseteq \R^n$, there are constants $\newC\label{C:10}=\oldC{C:10}(d,K)>1$ and $\newr\label{r:30}=\oldr{r:30}(d,K)>0$ such that
\begin{equation}
    \oldC{C:10}^{-1}\big(\lvert\Pi_x^V[y-x]\rvert+\lvert\Pi_x^W[y-x]\rvert^{\frac{1}{\eta}}\big)\underset{(A)}{\leq}d(x,y)\underset{(B)}{\leq} \oldC{C:10}\big(\lvert\Pi_x^V[y-x]\rvert+\lvert\Pi_x^W[y-x]\rvert^{\frac{1}{\eta}}\big)
    \label{eq:50},
\end{equation}
for every $x\in K$ and every $y\in B_d(x,\oldr{r:30})$.
\end{definition}

An immediate consequence of the above definition, is that for every fixed $\eta$, up to local equivalence, there exists a unique $\eta$-squeezed metric on $V$.

\begin{proposition}\label{prop:equiv}
If $d_1$ and $d_2$ are two $\eta$-squeezed metrics on $V$, then they are locally equivalent, i.e., for any $x\in\R^n$ there are an open neighbourhood $U$ of $x$ and a constant $C>1$ such that:
$$C^{-1}d_1(y,z)\leq d_2(y,z)\leq C d_1(y,z),$$
whenever $y,z\in U$.
\end{proposition}
\end{parag}



\subsection{Construction of squeezed metrics}

\begin{parag}[The $d_{V,\eta}$ metrics]\label{dVmtr} In order to show that Definition \ref{def:squmet} makes sense, we must construct $\eta$-squeezed metrics on $V$. To do so, fix an $\eta\in[1,2]$ and for any curve $\gamma\in \Lip([0,1],\R^n)$ define:
$$\ell_{\eta}(\gamma):=\sup_{t\in[0,1]} \max_{v\in \mathcal{D} \gamma(t)}\big(\lvert\Pi_{\gamma(t)}^V[v]\rvert+\lvert\Pi_{\gamma(t)}^W[v]\rvert^{\frac{1}{\eta}}\big),$$
where $\mathcal{D}\gamma(t)$ is the derived set of $\gamma$ at $t$, see Definition \eqref{def:derived}.
The anisotropic length $\ell_\eta$ induces a distance.

\begin{lemma}
Let $d_{V,\eta}:\R^n\times\R^n\to\R$ be the function defined as: 
$$d_{V,\eta}(x,y):=\inf\{\ell_{\eta}(\gamma):\gamma\in \Lip([0,1],\R^n),~\gamma(0)=x,~\gamma(1)=y\}.$$
Then $d_{V,\eta}$ is a metric, which we will refer to as the $\eta$\emph{-box distance}.
\end{lemma}

\begin{remark}
The definition of $\eta$-box distance is inspired by \cite[Definition 1.1]{nagel}.
\end{remark}

\begin{proof}
Since identity and symmetry are trivially satisfied, the only condition to verify in order to prove that $d_{V,\eta}$ is a distance, is the triangular inequality. Fix $x,y,z\in\R^n$,  $\epsilon>0$ and let $\gamma_1,\gamma_2\in\Lip([0,1],\R^n)$ be curves joining $x$ to $y$ and $y$ to $z$ respectively and suppose:
$$\ell_\eta(\gamma_1)\leq d(x,y)+\epsilon\qquad \text{and}\qquad\ell_\eta(\gamma_2)\leq d(y,z)+\epsilon.$$
Having defined $\lambda:=\big(\frac{\ell_\eta(\gamma_1)}{\ell_\eta(\gamma_1)+\ell_\eta(\gamma_2)}\big)^\eta\in (0,1)$, we let $\gamma\in\Lip([0,1],\R^n)$ be the concatenation of $\gamma_1$ and $\gamma_2$ defined as
$$\gamma(t):=\begin{cases}\gamma_1\Big(\frac{t}{\lambda}\Big) &\text{ if }t\in[0,\lambda],\\
\gamma_2\Big(\frac{t-\lambda}{1-\lambda}\Big) &\text{ if }t\in[\lambda,1].
\end{cases}$$
The curve $\gamma$ joins $x$ and $z$ and
$$d_{V,\eta}(x,z)\leq\ell_\eta(\gamma)\leq \max\bigg\{\frac{\ell_\eta(\gamma_1)}{\lambda^{1/\eta}},\frac{\ell_\eta(\gamma_2)}{(1-\lambda)^{1/\eta}}\bigg\}\leq \ell_\eta(\gamma_1)+\ell_\eta(\gamma_2),$$
where the last inequality comes from the definition of $\lambda$ and few algebraic computations. Thanks to the choice of $\gamma_1$ and $\gamma_2$, we have that: $$d_{V,\eta}(x,z)\leq d_{V,\eta}(x,y)+d_{V,\eta}(y,z)+2\epsilon,$$ which by arbitrariness of $\epsilon$, concludes the proof.
\end{proof}

\begin{lemma}
Let $d_{V,\eta} \colon \R^n \times \R^n \to \R$ be defined by
\[
d_{V,\eta}(x,y) := \inf\Bigl\{\ell_\eta(\gamma) : \gamma \in \Lip([0,1],\R^n),\; \gamma(0)=x,\; \gamma(1)=y\Bigr\}.
\]
Then $d_{V,\eta}$ is a metric on $\R^n$, which we call the \emph{$\eta$-box distance}.
\end{lemma}

\begin{remark}
The definition of the $\eta$-box distance is inspired by \cite[Definition 1.1]{nagel}.
\end{remark}

\begin{proof}
Since the properties of non-degeneracy and symmetry are trivially satisfied, it suffices to verify the triangle inequality.

Fix $x,y,z \in \R^n$ and $\epsilon>0$. Choose curves $\gamma_1,\gamma_2 \in \Lip([0,1],\R^n)$ joining $x$ to $y$ and $y$ to $z$, respectively, such that
\[
\ell_\eta(\gamma_1) \le d_{V,\eta}(x,y) + \epsilon \quad \text{and} \quad \ell_\eta(\gamma_2) \le d_{V,\eta}(y,z) + \epsilon.
\]
Define
$\lambda := (\frac{\ell_\eta(\gamma_1)}{\ell_\eta(\gamma_1)+\ell_\eta(\gamma_2)})^\eta \in (0,1)$ and consider the concatenated curve $\gamma \in \Lip([0,1],\R^n)$ defined by
\[
\gamma(t) := \begin{cases}
\gamma_1\Bigl(\frac{t}{\lambda}\Bigr) & \text{if } t \in [0,\lambda],\\[1mm]
\gamma_2\Bigl(\frac{t-\lambda}{1-\lambda}\Bigr) & \text{if } t \in [\lambda,1].
\end{cases}
\]
Then $\gamma$ joins $x$ to $z$. Moreover, one can show that
\[
\ell_\eta(\gamma) \le \max\Biggl\{ \frac{\ell_\eta(\gamma_1)}{\lambda^{1/\eta}},\; \frac{\ell_\eta(\gamma_2)}{(1-\lambda)^{1/\eta}} \Biggr\}.
\]
By definition, of $\lambda$ we have $\ell_\eta(\gamma_1)/\lambda^{1/\eta} = \ell_\eta(\gamma_1)+\ell_\eta(\gamma_2)$. Moreover, since $t\mapsto t^{1/\eta}$ is concave, we have $(1-\lambda)^{1/\eta} \ge 1-\lambda^{1/\eta}$, which implies that $\ell_\eta(\gamma_2)/(1-\lambda)^{1/\eta} \le \ell_\eta(\gamma_1)+\ell_\eta(\gamma_2)$. Therefore, 
\[
\ell_\eta(\gamma) \le \max\Bigl\{\frac{\ell_\eta(\gamma_1)}{\lambda^{1/\eta}},\,\frac{\ell_\eta(\gamma_2)}{(1-\lambda)^{1/\eta}}\Bigr\} \le \ell_\eta(\gamma_1)+\ell_\eta(\gamma_2).
\]
It follows that
\[
d_{V,\eta}(x,z) \le \ell_\eta(\gamma) \le \ell_\eta(\gamma_1)+\ell_\eta(\gamma_2) \le \bigl(d_{V,\eta}(x,y)+d_{V,\eta}(y,z)+2\epsilon\bigr).
\]
Since $\epsilon>0$ is arbitrary, the triangle inequality holds, and the proof is complete.
\end{proof}

The following proposition relates the $\eta$-box distance to the Euclidean metric, showing that they are H\"older equivalent.

\begin{proposition}\label{prop:anis}
For every $x\in\R^n$, one has
\begin{equation}
\begin{split}
\frac{2}{3}\lvert x-y\rvert \le d_{V,\eta}(x,y) \le 2\lvert x-y\rvert^{1/\eta} \quad &\text{for all } y\in U(x,4^{-\eta}),\\[1mm]
\lvert x-y\rvert \le d_{V,\eta}(x,y) \quad &\text{for all } y\in B_{d_{V,\eta}}(x,1).
\end{split}
\end{equation}
\end{proposition}

\begin{proof}
We first prove the upper bound. Let $\sigma\colon [0,1]\to\R^n$ be the straight-line segment defined by $\sigma(t)=x+t(y-x)$.
Then, by the definition of the $\eta$-length,
\[
\ell_\eta(\sigma) \le \lvert\Pi^V_{\sigma(t)}(y-x)\rvert + \lvert\Pi^W_{\sigma(t)}(y-x)\rvert^{1/\eta}.
\]
For $y\in U(x,4^{-\eta})$, one easily verifies that
\[
\lvert\Pi^V_{\sigma(t)}(y-x)\rvert + \lvert\Pi^W_{\sigma(t)}(y-x)\rvert^{1/\eta} \le 2\lvert y-x\rvert^{1/\eta}.
\]
Since $d_{V,\eta}(x,y)$ is the infimum of $\eta$-lengths over all curves joining $x$ and $y$, this shows that
\[
d_{V,\eta}(x,y) \le 2\lvert y-x\rvert^{1/\eta}.
\]

Next, we establish the lower bound. Let $\gamma\in \Lip([0,1],\R^n)$ be a curve joining $x$ and $y$ such that
\[
\frac{2}{3}\,\ell_\eta(\gamma) \le d_{V,\eta}(x,y).
\]
The previous step implies that, when $\lvert x-y\rvert\le 4^{-\eta}$, one has $d_{V,\eta}(x,y)\le 1/2$. Hence, $\ell_\eta(\gamma)\le 3/4$. In particular, for every $t\in[0,1]$ we have that for every $v\in \mathcal{D}\gamma(t)$ we have
\begin{equation}\label{eq:2000}
\lvert\Pi^V_{\sigma(t)}[v]\rvert + \lvert\Pi^W_{\sigma(t)}[v]\rvert^{1/\eta} \le 1.
\end{equation}
Since any vector decomposes as $v=\Pi^V_{\sigma(t)}[v]+\Pi^W_{\sigma(t)}[v]$, the triangle inequality yields
\[
\lvert v\rvert \le \lvert\Pi^V_{\sigma(t)}[v]\rvert + \lvert\Pi^W_{\sigma(t)}[v]\rvert\le \lvert\Pi^V_{\sigma(t)}[v]\rvert + \lvert\Pi^W_{\sigma(t)}[v]\rvert^{1/\eta},
\]
where the last inequality follows from \eqref{eq:2000}.
Thus, by \eqref{eq:2000} we have $\lvert v\rvert\le 1$ for every $v\in\mathcal{D}\gamma(t)$ and every $t\in[0,1]$. It follows that
\[
\lvert x-y\rvert \le \int_0^1 \lvert \dot{\gamma}(t)\rvert\,dt \le \sup_{t\in[0,1]} \max_{v\in\mathcal{D}\gamma(t)} \lvert v\rvert \le \ell_\eta(\gamma) \le \frac{3}{2}\,d_{V,\eta}(x,y),
\]
or in other words
\[
\frac{2}{3}\lvert x-y\rvert \le d_{V,\eta}(x,y).
\]
The proof of the final inequality follows by an analogous argument.
\end{proof}

In the next proposition it will be important to compare balls centered at different points. Since the vector fields $X_1,\dots,X_n$ are smooth, the maps
\[
x\mapsto \Pi_x^V \quad\text{and}\quad x\mapsto \Pi_x^W
\]
are smooth. Hence, for every $R=R(x)>0$ small enough there exists a constant $C(R)>1$ such that
\begin{equation}
    \|\Pi_y^V-\Pi_z^V\|\le C(R)|y-z| \,\text{ and }\, \|\Pi_y^W-\Pi_z^W\|\le C(R)|y-z|,
    \label{eq:41}
\end{equation}
for every $y,z\in B(x,R)$.

\begin{proposition}\label{prop:loc:equiv}
For any $\eta\in[1,2)$ the metric $d_{V,\eta}$ is $\eta$-squeezed on $V$. In particular, for any $R>0$ and $x\in U(0,R)$, one has
\begin{equation}\label{eq:ineq}
\frac{1}{1+C(R)}\Bigl(|\Pi_x^V[y-x]|+|\Pi_x^W[y-x]|^{1/\eta}\Bigr)
\le d_{V,\eta}(x,y)\le 2\Bigl(|\Pi_x^V[y-x]|+|\Pi_x^W[y-x]|^{1/\eta}\Bigr),
\nonumber
\end{equation}
whenever 
$y\in U(x,\newr\label{r:20})\cap U(0,R)$, with
\[
\oldr{r:20}:=\min\{(4C(R))^{-\eta-1},(36C(R))^{-1/(2-\eta)}\}.
\]
\end{proposition}

\begin{proof}
Throughout the proof we assume that $x,y\in U(0,R)$ and $d_{V,\eta}(x,y)\le r_{20}$.

\textbf{Upper bound.} Define the straight-line curve $\sigma(t)=x+t(y-x)$, and set $u:=y-x$. By definition of $d_{V,\eta}$ we have
\[
d_{V,\eta}(x,y)\le \sup_{t\in[0,1]}\Bigl(|\Pi_{\sigma(t)}^V[u]|+|\Pi_{\sigma(t)}^W[u]|^{1/\eta}\Bigr).
\]
Since
\[
|\Pi_{\sigma(t)}^V[u]|\le |\Pi_x^V[u]|+C(R)|u|^2,\qquad\text{and}\qquad
|\Pi_{\sigma(t)}^W[u]|\le |\Pi_x^W[u]|+C(R)|u|^2,
\]
we deduce
\[
d_{V,\eta}(x,y)\le |\Pi_x^V[u]|+C(R)|u|^2+\Bigl(|\Pi_x^W[u]|+C(R)|u|^2\Bigr)^{1/\eta}.
\]
Thanks to Lemma~\ref{prop:anis} and to the choice of $u$, one obtains
\[
d_{V,\eta}(x,y)\le 2\Bigl(|\Pi_x^V[u]|+|\Pi_x^W[u]|^{1/\eta}\Bigr),
\]
which is the upper bound in \eqref{eq:ineq}.

\textbf{Lower bound.} Let $\gamma\in\Lip([0,1],\R^n)$ be a curve joining $x$ to $y$ such that
\[
\ell_\eta(\gamma)\le \frac{2}{3}\,d_{V,\eta}(x,y).
\]
Then, by definition of $d_{V,\eta}$, for every $s\in[0,1]$ and every $v\in\mathcal{D}\gamma(s)$ we have
\begin{equation}\label{eq:1003}
|\Pi_{\gamma(s)}^V[v]|\le \frac{2}{3}\,d_{V,\eta}(x,y) \quad\text{and}\quad
|\Pi_{\gamma(s)}^W[v]|\le \Bigl(\frac{2}{3}\,d_{V,\eta}(x,y)\Bigr)^\eta.
\end{equation}
In particular, these inequalities imply that $|v|\le d_{V,\eta}(x,y)$, and  Proposition~\ref{prop:mean} yields
\[
|\gamma(s)-x|\le d_{V,\eta}(x,y)|s| \quad\text{for all } s\in[0,1].
\]
Moreover, again by Lemma~\ref{prop:anis}, we have
\begin{equation}\label{eq:1002}
|\Pi_x^V[y-x]|\le |y-x|+|\Pi_x^W[y-x]|\le d_{V,\eta}(x,y)+|\Pi_x^W[y-x]|.
\end{equation}
It now suffices to show that
\begin{equation}\label{eq:claim}
|\Pi_x^W[y-x]|\le (1+C(R))\,d_{V,\eta}(x,y)^\eta.
\end{equation}
Indeed, combining \eqref{eq:1002} and \eqref{eq:claim} yields
\[
|\Pi_x^V[y-x]|+|\Pi_x^W[y-x]|^{1/\eta}\le d_{V,\eta}(x,y)+(1+C(R))^{1/\eta}\,d_{V,\eta}(x,y)
\]
so that
\[
d_{V,\eta}(x,y)\ge \frac{1}{1+C(R)}\Bigl(|\Pi_x^V[y-x]|+|\Pi_x^W[y-x]|^{1/\eta}\Bigr).
\]

To prove \eqref{eq:claim}, define the Lipschitz function
\[
f(s)=\Bigl\langle \Pi_x^W\bigl[\gamma(s)-x-s(y-x)\bigr],\frac{\Pi_x^W[y-x]}{|\Pi_x^W[y-x]|}\Bigr\rangle, \quad s\in[0,1].
\]
Since $f(0)=f(1)=0$, Proposition~\ref{prop:mean} implies the existence of $t\in[0,1]$ such that $0\in\mathcal{D}f(t)$. Consequently, there exist sequences $a_i,b_i\to t$ with
\[
\lim_{i\to\infty}\frac{f(b_i)-f(a_i)}{b_i-a_i}=0,
\]
and, by the definition of the derived set,
\[
\lim_{i\to\infty}\frac{\gamma(b_i)-\gamma(a_i)}{b_i-a_i}=w\in\mathcal{D}\gamma(t).
\]
By continuity of the scalar product, we infer
\begin{equation}\label{eq:2001}
|\Pi_x^W[y-x]|=\Bigl\langle \Pi_x^W(w),\frac{\Pi_x^W[y-x]}{|\Pi_x^W[y-x]|}\Bigr\rangle.
\end{equation}
Using \eqref{eq:1003}, the bound $|w|\le d_{V,\eta}(x,y)$, and the Lipschitz property of $\Pi^W$, we obtain
\[
|\Pi_x^W[y-x]|\le |\Pi_{\gamma(t)}^W(w)|+C(R)|\gamma(t)-x|\,|w|\le d_{V,\eta}(x,y)^\eta+C(R)d_{V,\eta}(x,y)^2.
\]
Since $d_{V,\eta}(x,y)<1$, the claim \eqref{eq:claim} follows. This completes the proof.
\end{proof}

We further assume that we can find additional vector fields 
$X_{k+1},\ldots,X_n$,
which are commutators of degree 2 of elements of \(\mathcal{X}\) such that for every \(x\in\R^n\) the family
$\{X_1(x),\ldots,X_n(x)\}$
is linearly independent. We then set
$$
W(x)=\Span\{X_{k+1}(x),\ldots,X_n(x)\}.
$$

\begin{proposition}\label{inclusioncc}
The \cc metric \(d_V\) associated to \(V\), as introduced in \S\ref{ccdistance}, is \(2\)-squeezed on \(V\). In particular, the metrics \(d_V\) and \(d_{V,2}\) are locally equivalent.
\end{proposition}

\begin{proof}
Let \(R>0\), fix \(\zeta\in U(0,R)\), and consider a compact set \(K\Subset U_\zeta\), where \(U_\zeta\) is the open neighbourhood introduced in \S\ref{par:BBT}. Suppose that 
$$
x\in K\cap U(0,R) \quad \text{and} \quad y\in B_{d_V}(x,\min\{\oldr{r:4},\oldr{r:7},1\})\cap U(0,R).
$$
For notational convenience, define
\[
\begin{aligned}
M_1(\zeta)&:= \max_{\substack{1\le i\le n\\ z\in U_\zeta}} |X_i(z)|,\quad
M_2(\zeta) := \max_{z\in U_\zeta} \|\mathrm{d}^2\Exp_\zeta\|,\\[1mm]
m_V(\zeta)&:= \min_{\substack{v\in\mathbb{S}^{k-1}\\ z\in U_\zeta}} \Bigl|\sum_{i=1}^k v_i\,X_i(z)\Bigr|,\quad
m_W(\zeta) := \min_{\substack{v\in\mathbb{S}^{n-k-1}\\ z\in U_\zeta}} \Bigl|\sum_{i=k+1}^n v_i\,X_i(z)\Bigr|.
\end{aligned}
\]
By assumption, these quantities are positive for every \(R>0\). By Theorem~\ref{BallBox}, since
$$
y\in B_{d_V}(x,2d_V(x,y))\setminus B_{d_V}(x,d_V(x,y)/2),
$$
there exists \(t\in Q(2\oldC{C:1}\,d_V(x,y))\setminus Q(d_V(x,y)/2\oldC{C:1})\) such that 
$y=\Exp_x(t)$.
By the definition of the anisotropic box \(Q(\cdot)\), for each \(i=1,\dots,n\) we have
\begin{align}
|t_i|&\le (2\oldC{C:1}\,d_V(x,y))^{\deg X_i},\label{eq:1040}\\[1mm]
\Bigl(\frac{d_V(x,y)}{2\oldC{C:1}}\Bigr)^{\deg X_j}&\le |t_j|\quad\text{for some } j\in\{1,\dots,n\}.\label{eq:1042}
\end{align}

\textbf{Upper bound.} Using the expansion of the exponential, see \eqref{eq:10}, valid when \(d_V(x,y)\le \oldr{r:4}\), we have
$$
y=\Exp_x(t)=x+\sum_{i=1}^n t_iX_i(x)+\mathrm{d}^2\Exp_x(s\,t)[t,t]
$$
for some \(s\in[0,1]\). Hence, by the triangle inequality,
\begin{equation}\label{eq:1041}
\begin{split}
|\Pi_x^V[y-x]| &\le \sum_{i=1}^k |t_i|\,|X_i(x)| + \Bigl|\Pi_x^V\Bigl(\mathrm{d}^2\Exp_x(s\,t)[t,t]\Bigr)\Bigr|\\[1mm]
&\le 2k\,\oldC{C:1}\,M_1(\zeta)\,d_V(x,y)+16n\,\oldC{C:1}^4\,M_2(\zeta)\,d_V(x,y)^2\\[1mm]
&\le 32n\,\oldC{C:1}^4\bigl(M_1(\zeta)+M_2(\zeta)\bigr)d_V(x,y).
\end{split}
\end{equation}
Similarly, one obtains
\begin{equation}\label{eq:1050}
\begin{split}
|\Pi_x^W[y-x]| &\le \sum_{i=k+1}^n |t_i|\,|X_i(x)| + \Bigl|\Pi_x^W\Bigl(\mathrm{d}^2\Exp_x(s\,t)[t,t]\Bigr)\Bigr|\\[1mm]
&\le 4(n-k)\,\oldC{C:1}^2\,M_1(\zeta)\,d_V(x,y)^2+16n\,\oldC{C:1}^4\,M_2(\zeta)\,d_V(x,y)^2\\[1mm]
&\le 32n\,\oldC{C:1}^4\bigl(M_1(\zeta)+M_2(\zeta)\bigr)d_V(x,y)^2.
\end{split}
\end{equation}
Thus, the anisotropic norm
$$
|\Pi_x^V[y-x]|+|\Pi_x^W[y-x]|^{1/2}
$$
is bounded by a constant times \(d_V(x,y)\), which gives the desired upper bound.

\medskip

\textbf{Lower bound.} Since \(d_V(x,y)<1\), one may estimate 
$$
|t|^2\le 4k\,\oldC{C:1}^2\,d_V(x,y)^2+16(n-k)\,\oldC{C:1}^4\,d_V(x,y)^4\le 16n\,\oldC{C:1}^4\,d_V(x,y)^2.
$$
Define
$$
L:=\frac{2\bigl(m_V(\zeta)+k\,M_2(\zeta)\bigr)}{m_W(\zeta)}
$$
and assume that
\begin{equation}
d_V(x,y)\le \frac{m_V(\zeta)}{128\,n\,\oldC{C:1}^6\,M_2(\zeta)\,L}.
    \label{bounddvxy}
\end{equation}
We now distinguish two cases.

\medskip

\textbf{Case 1.} Suppose there exists some \(j\in\{1,\dots,k\}\) with
$$
|t_j|\ge \frac{d_V(x,y)}{2\oldC{C:1}L}.
$$
Then, using the expansion \eqref{eq:10}, we have
\begin{equation}\label{eq:A1}
|\Pi_x^V[y-x]|\ge m_V(\zeta)|t|-M_2(\zeta)|t|^2
\ge \frac{m_V(\zeta)d_V(x,y)}{4\oldC{C:1}L},
\end{equation}
where the first inequality above comes from the definition of $m_V(\zeta)$ and $M_2(\zeta)$ and the last one follows from the choice of $t$ and on \eqref{bounddvxy}.
\medskip

\textbf{Case 2.} If 
$$
|t_i|<\frac{d_V(x,y)}{2\oldC{C:1}L}\quad\text{for all }i=1,\dots,k,
$$
then one can refine the estimate for \(|t|^2\) to obtain
$$
|t|^2< \frac{k\,d_V(x,y)^2}{4\oldC{C:1}^2L}+16(n-k)\,\oldC{C:1}^4\,d_V(x,y)^4
\le \frac{m_W(\zeta)d_V(x,y)^2}{8\oldC{C:1}^2M_2(\zeta)}.
$$
Moreover, by \eqref{eq:1042} there is some \(j\in\{k+1,\dots,n\}\) such that
$$
\Bigl(\frac{d_V(x,y)}{2\oldC{C:1}}\Bigr)^2\le |t_j|.
$$
Thus, using \eqref{eq:10} for the \(W\)-component we infer
\begin{equation}\label{eq:A2}
|\Pi_x^W[y-x]|\ge m_W(\zeta)|t|-M_2(\zeta)|t|^2
\ge \frac{m_W(\zeta)d_V(x,y)^2}{8\oldC{C:1}^2}.
\end{equation}
Taking square roots in \eqref{eq:A2} gives
$$
|\Pi_x^W[y-x]|^{1/2}\ge \Bigl(\frac{m_W(\zeta)}{8\oldC{C:1}^2}\Bigr)^{1/2}d_V(x,y).
$$
Combining the estimates from Cases 1 and 2 shows that
$$
|\Pi_x^V[y-x]|+|\Pi_x^W[y-x]|^{1/2}\ge \frac{1}{1+C}\,d_V(x,y)
$$
for a constant \(C\) (depending on \(\zeta\) and the above quantities). 
A standard covering argument then shows that these local estimates yield the local equivalence of \(d_V\) and \(d_{V,2}\) on any compact set and completes the proof.
\end{proof}

\begin{proposition}\label{limit}
Let \(\eta_0\in[1,2]\). Then, for every \(x,y\in\R^n\),
\[
\lim_{\eta\to\eta_0} d_{V,\eta}(x,y) = d_{V,\eta_0}(x,y).
\]
\end{proposition}

\begin{proof}
By the definition of \(d_{V,\eta}\), for every \(x,y\in\R^n\) we have
\[
d_{V,\eta}(x,y)\le |x-y| + |x-y|^{1/\eta}.
\]
Since \(1/\eta\in[1/2,1]\) for \(\eta\in[1,2]\), it follows that
\begin{equation}\label{eq:2002}
d_{V,\eta}(x,y)\le 2|x-y| + |x-y|^{1/2}.
\end{equation}
For notational convenience, set
\[
C(x,y):=2\max\Bigl\{1,\;2|x-y| + |x-y|^{1/2}\Bigr\}.
\]

Let \(\epsilon,\delta>0\). For each \(\eta\in[\eta_0-\epsilon,\eta_0+\epsilon]\), choose a curve \(\gamma_\eta\in\Lip([0,1],\R^n)\) such that
\[
\ell_\eta(\gamma_\eta)\le d_{V,\eta}(x,y) + \delta.
\]
Note that by \eqref{eq:2002} we have that \(d_{V,\eta}(x,y)\le C(x,y)\).

Now, let \(\eta,\eta'\in[\eta_0-\epsilon,\eta_0+\epsilon]\). Since \(\gamma_\eta\) is admissible for \(d_{V,\eta'}\), we have
\[
d_{V,\eta'}(x,y)\le \ell_{\eta'}(\gamma_\eta).
\]
For every \(t\in[0,1]\) and any \(v\in\mathcal{D}\gamma_\eta(t)\), we can write
\[
|\Pi_{\gamma_\eta(t)}^W[v]|^{1/\eta'} 
= |\Pi_{\gamma_\eta(t)}^W[v]|^{1/\eta} \, |\Pi_{\gamma_\eta(t)}^W[v]|^{1/\eta' - 1/\eta}=|\Pi_{\gamma_\eta(t)}^W[v]|^{1/\eta} \, (|\Pi_{\gamma_\eta(t)}^W[v]|^{1/\eta})^\frac{\eta-\eta'}{\eta'}.
\]
Notice that 
\begin{equation}
    |\Pi_{\gamma_\eta(t)}^W[v]|^{1/\eta}\leq \ell_\eta(\gamma_\eta)\leq d_{V,\eta}(x,y)+\delta \leq C(x,y)+\delta,
    \nonumber
\end{equation}
where in the last inequality we used \eqref{eq:2002}. Thanks to the above computation, we deduce that
\begin{equation}
    \label{eq:cont1}
    \begin{split}
&\qquad\qquad|\Pi_{\gamma_\eta(t)}^W[v]|^{1/\eta'} \le (C(x,y)+\delta)^{|\eta-\eta'|/\eta'} \, |\Pi_{\gamma_\eta(t)}^W[v]|^{1/\eta}\\
\leq& (C(x,y)+\delta)^{|\eta-\eta'|/\eta'} \,\ell_\eta(\gamma_\eta)\leq (C(x,y)+\delta)^{|\eta-\eta'|/\eta'}(d_{V,\eta}(x,y)+\delta).
    \end{split}
\end{equation}
Since \(\delta>0\) is arbitrary, we obtain
\[
d_{V,\eta'}(x,y) \le C(x,y)^{|\eta-\eta'|/\eta'} \, d_{V,\eta}(x,y).
\]
Exchanging the roles of \(\eta\) and \(\eta'\) yields
\[
d_{V,\eta}(x,y) \le C(x,y)^{|\eta-\eta'|/\eta} \, d_{V,\eta'}(x,y).
\]
In particular
\[
C(x,y)^{-|\eta-\eta_0|/\eta_0} \, d_{V,\eta_0}(x,y) \le d_{V,\eta}(x,y) \le C(x,y)^{|\eta-\eta_0|/\eta} \, d_{V,\eta_0}(x,y).
\]
Since the exponents \(|\eta-\eta_0|/\eta\) and \(|\eta-\eta_0|/\eta_0\) tend to zero as \(\eta\to\eta_0\), the result follows by the squeeze theorem.
\end{proof}

\end{parag}

\begin{parag}[Rectifiability properties of squeezed metrics]

Suppose $\eta\in[1,2]$ and let $d$ be an $\eta$-squeezed metric on $V$. Since the metrics $d$ and $d_{V,\eta}$ are locally equivalent as they are supposed to be both $\eta$-squeezed, we have that for any $\alpha\in[0,\infty)$ the measures $\Haus^\alpha_d$ and $\Haus^\alpha_{d_{V,\eta}}$ are mutually absolutely continuous.

\begin{proposition}\label{UIAM}
Suppose $
f:K\Subset\R^m\to (\R^n,d_{V,\eta})
$
is an \(L\)-Lipschitz map. Then \(f\) is locally Lipschitz with respect to the Euclidean metric and
\[
\mathscr{H}^m_{d_{V,\eta}}\Bigl(\bigl\{x\in K:\operatorname{rk}(Df(x))<m\bigr\}\Bigr)=0.
\]
\end{proposition}

\begin{proof}
Since \(f\) is \(L\)-Lipschitz with respect to \(d_{V,\eta}\), Lemma~\ref{prop:anis} shows that \(f\) is locally Lipschitz in the Euclidean sense.  By Euclidean Rademacher’s theorem we know that $f$ is differentiable $\Leb^m$-almost everywhere.
Fix a point \(x\in K\) at which $f$ is differentiable and suppose there exists \(w\in\mathbb{S}^{m-1}\) with
\[
Df(x)[w]=0.
\]
We claim that for $\Leb^m$-almost every $x\in K$ we have
\begin{equation}\label{eq:1012}
\lim_{\substack{t\to 0 \\ x+tw\in K}} \frac{d_{V,\eta}\bigl(f(x+tw),f(x)\bigr)}{t}=0.
\end{equation}
Fix a point $x\in K$ and given $w\in \mathbb{S}^{m-1}$ we let $t\in \R$ be such that 
$x+tw\in K$. This implies by Proposition \ref{prop:loc:equiv} that 
\begin{equation}
    \begin{split}
    \frac{d_{V,\eta}\bigl(f(x+tw),f(x)\bigr)}{t}\leq 2\frac{|\Pi_x^V[f(x+tw)-f(x)]|}{t}+2\Big(\frac{|\Pi_x^W[f(x+tw)-f(x)]|}{t^\eta}\Big)^{1/\eta}.
        \nonumber
    \end{split}
\end{equation}
 Note that by Lebesgue's differentiation theorem and by Tonelli's theorem, we know that for $\Leb^m$-almost every $x\in K$ we have that $x$ is a $\Haus^1$-density point in $K\cap x+\mathrm{span}(w)$.
Therefore, from now on $x\in K$ will be a fixed point such that $f$ is differentiable at $x$ and 
$$\lim_{r\to 0}\frac{\Haus^1(K\cap x+\ell_w\cap B(x,r))}{2r}=1,$$
where $\ell_w$ denotes the line spanned by $w$. 
Thanks to our choice of $x$, for every $j\in \N$ there exists $\delta=\delta(x,j)>0$ such that whenever $\lvert t\rvert\leq \delta$ is such that $x+tw\in K$, we have
$$\frac{\lvert f(x+tw)-f(x)-Df(x)[tw]\rvert }{\lvert t\rvert}\leq \frac{1}{j}.$$
This implies that if $\lvert t\rvert\leq \delta$, then
\begin{equation}
    \begin{split}
\frac{d_{V,\eta}\bigl(f(x+tw),f(x)\bigr)}{t}\leq 2j^{-1}+2\Big(\frac{|\Pi_x^W[f(x+tw)-f(x)]|}{t^\eta}\Big)^{1/\eta}. 
    \end{split}
\end{equation}
For every $x\in K$ and $t>0$ such that $x+tw\in K$ we let
\begin{equation}
    \begin{split}
    \mathrm{II}_w(x,t):=\frac{|\Pi_x^W[f(x+tw)-f(x)]|}{t^\eta}.
    \nonumber
    \end{split}
\end{equation}
In order to conclude the proof of \eqref{eq:1012}, thanks to the arbitrariness of $j$ we just need to prove that 
\begin{equation}
\lim_{\substack{t\to 0 \\ x+tw\in K}} \mathrm{II}_w(x,t)=0.
    \label{IIgoto0}
\end{equation}
For clarity, we prove \eqref{IIgoto0} for \(t>0\); the case \(t<0\) is analogous. Let \(t>0\). Since \(K\) is compact, the complement of 
\[
\mathcal{K}(t):=\{s\in[0,t]:x+sw\in K\}
\]
is a disjoint union of countably many open intervals \(\{(a_j,b_j)\}_{j\in\N}\). Moreover, by our choice of \(x\), for every \(0<\epsilon<1\) there exists \(r(\epsilon)>0\) such that if \(0<t<r(\epsilon)\), then
\[
\sum_{j\in\N}(b_j-a_j)\leq \epsilon t.
\]

Suppose that \(\tau:=\{t_1,\ldots, t_M\}\) is a finite \(\epsilon^3t\)-dense subset of \(\mathcal{K}(t)\). Notice that if \(t\) is small enough, then \eqref{eq:41} together with Proposition \ref{prop:loc:equiv} imply that 
\begin{equation}
    \begin{split}
        \Bigl|\Bigl(\Pi_{x+t_jw}^W-\Pi_{x}^W\Bigr)[f(x+t_{j+1}w)-f(x+t_{j}w)]\Bigr| &\leq C(R)|t_j|\,|f(x+t_{j+1}w)-f(x+t_{j}w)|\\[1mm]
        &\leq  C(R)L|t_j|\,|t_{j+1}-t_{j}|.
        \label{eq:1029}
    \end{split}
\end{equation}
In addition, if \(t\) is chosen small enough, then by Proposition \ref{prop:loc:equiv} and \eqref{eq:41} we have 
\begin{equation}
| \Pi_x^W[f(x+t_{j+1}w)-f(x+t_{j}w)]|\leq  (1+C(R))^\eta L^\eta |t_{j+1}-t_j|^\eta,
    \label{eq:stimeschiacc}
\end{equation}
where \(R\) is the radius given by \eqref{eq:41} relative to \(x\). 

The estimates \eqref{eq:1029} and \eqref{eq:stimeschiacc} allow us to bound \(\mathrm{II}_w(x,t)\) as follows:
\begin{equation}
\begin{split}
    &\mathrm{II}_w(x,t)\leq t^{-\eta}\sum_{j=1}^{M-1}\bigl|\Pi_{x}^W[f(x+t_{j+1}w)-f(x+t_{j}w)]\bigr|\\[1mm]
&\leq t^{-\eta}\sum_{j=1}^{M-1}\Bigl|\Pi_{x+t_jw}^W[f(x+t_{j+1}w)-f(x+t_{j}w)]\Bigr|+C(R)L\,t^{-\eta}\sum_{j=1}^{M-1} t_j|t_{j+1}-t_{j}|\\[1mm]
&\leq (1+C(R))^\eta L^\eta t^{-\eta}\sum_{j=1}^{M-1} |t_{j+1}-t_{j}|^\eta+C(R)L\, \varepsilon\, t^{2-\eta}\\[1mm]
&\leq (1+C(R))^\eta L^\eta \varepsilon^{\eta-1}+C(R)L\, \varepsilon\, t^{2-\eta}.
\nonumber
\end{split}
\end{equation}
Since \(\varepsilon>0\) is arbitrary, this completes the proof of \eqref{IIgoto0} and thus of \eqref{eq:1012}.

To conclude the proof of the proposition, let \(\iota:(\R^n,d_{V,\eta})\hookrightarrow \ell^\infty\) be an isometric embedding of the separable metric space \((\R^n,d_{V,\eta})\) into \(\ell^\infty\). Define \(F:=\iota\circ f\). Then \eqref{eq:1012} shows that if \(\operatorname{rk}(Df(x))<m\) at some \(x\in K\), there exists \(u\in\mathbb{S}^{m-1}\) such that the metric differential \(MD(F,x)[u]=0\). By Theorem~\ref{metricrademacher}, at \(\Leb^m\)-almost every \(x\in K\) the metric differential \(MD(F,x)\) exists and is a seminorm; consequently, Corollary~\ref{cor:rk} implies that
\[
\Haus^m_{\|\cdot\|_\infty}\bigl(F(\mathcal{N})\bigr)=0,
\]
where \(\mathcal{N}=\{x\in K:\operatorname{rk}(Df(x))<m\}\). Since \(\iota\) is an isometry, we deduce that
\[
\Haus^m_{d_{V,\eta}}\bigl(f(\mathcal{N})\bigr)=0.
\]
This completes the proof.
\end{proof}

\begin{proposition}
For any $m>k$ and $\eta\in(1,2]$, the metric space $(\R^n,d_{V,\eta})$ is $(\Haus^m,m)$-purely unrectifiable.
\end{proposition}

\begin{proof}
Let $f:K\Subset \R^m\to (\R^n,d_{V,\eta})$ be a Lipschitz map. With the same argument used to prove inclusion \eqref{eq:25}, one can show that $\text{im}(Df(x))\subseteq V(f(x))$ for $\Leb^m$-almost every $x\in K$. Since $m>k$, we deduce that $\text{rk}(Df)\leq k<m$ $\Leb^m$-almost everywhere on $K$ and thus Proposition \ref{UIAM} implies that $\Haus^m_{d_{V,\eta}}(f(K))=0$. The arbitrariness of $f$ proves the claim.
\end{proof}

\begin{proposition}[{\cite{AMMFrobenius3}}]\label{lusin}
Let $\vartheta,\varepsilon>0$ and suppose that 
\[
F:\Omega\times \mathbb{R}^{n-k}\to \mathbb{R}^{k\times (n-k)}
\]
is a locally Lipschitz map. Then there exist a compact set $\mathfrak{C}\subset \Omega$ and a function 
$u:\Omega\to \mathbb{R}^{n-k}$
of class $C^{1,\alpha}(\Omega)$ for every $\alpha\in (0,1)$, such that
\begin{enumerate}[label=(\roman*)]
    \item $\mathscr{L}^k\bigl(\Omega\setminus \mathfrak{C}\bigr)\le \varepsilon\, \mathscr{L}^k(\Omega)$;
    \item $\|u\|_{\infty}\le \vartheta$;
    \item $Du(x)=F\bigl(x,u(x)\bigr)$ for every $x\in \mathfrak{C}$.
\end{enumerate}
\end{proposition}

\begin{proposition}
If the metric $d$ is $\eta$-squeezed on $V$ for some $\eta\in (1,2)$, then there exists a Lipschitz function 
$f:K\to (\mathbb{R}^n,d)$
with $K\Subset \mathbb{R}^k$, such that
\[
\Haus^k_{d}\bigl(f(K)\bigr)>0.
\]
On the other hand, if $\eta=2$ and $V$ is non-involutive, then the metric space $(\mathbb{R}^n,d)$ is $k$-purely unrectifiable.
\end{proposition}

\begin{proof}
Pick any point $z\in \mathbb{R}^n$ and, without loss of generality, assume that
\[
V(z)=\mathrm{span}(e_1,\ldots,e_k)\cong \mathbb{R}^k.
\]
Choose $\rho>0$ small enough so that for every $w\in B(z,\rho)$ there exists a map
\[
F:B(z,\rho)\to \mathbb{R}^{k\times (n-k)}
\]
with the property that
$\operatorname{gr}\bigl(F(w)\bigr)=V(w)$.
Then, by Proposition~\ref{lusin}, there exists a compact set $C\subset \mathbb{R}^k$ of positive Lebesgue measure and a function
$f:C\to \mathbb{R}^{n-k}$
such that the graph of $f$, $\operatorname{gr}(f)$, is compactly contained in $B(z,\rho)$ and 
\[
\operatorname{Tan}\bigl(\operatorname{gr}(f),w\bigr)=V(w)
\]
for $\Haus^k$-almost every $w\in \operatorname{gr}(f)$. It can be easily checked that such function $u$ is Lipschitz (as here is supposed strictly smaller than $2$) for the $\eta$-squeezed metric $d$ thus yielding the first part of the proposition. 

If $\eta=2$, one can argue as in Proposition~\ref{whitney} to deduce that any Lipschitz image of a $k$-dimensional compact set is $C^{1,1}$-rectifiable. However, if $V$ is non-involutive, then by \cite[Theorem~1.3]{Balogh2011SizeOT} the set where a $C^2$ surface is tangent to a non-involutive distribution of dimension $k$ is $\Haus^k$-null. This implies that $(\mathbb{R}^n,d)$ is $k$-purely unrectifiable.
\end{proof}

\end{parag}

    %
    %
    %
\printbibliography

    %
    %
    %
    %
\vskip .5 cm
{\parindent = 0 pt\footnotesize
G.A.
\par
\smallskip
Dipartimento di Matematica,
Universit\`a di Pisa
\par
largo Pontecorvo 5,
56127 Pisa,
Italy
\par
\smallskip
e-mail: \texttt{giovanni.alberti@unipi.it}

\bigskip
A.Ma.
\par
\smallskip

Dipartimento di Matematica ``Tullio Levi-Civita'',
Universit\`a di Padova
\par
via Trieste 63, 
35121 Padova,
Italy
\par
\smallskip
e-mail: \texttt{annalisa.massaccesi@unipd.it}

\bigskip
A.Me.
\par
\smallskip
Universidad del Pais Vasco (UPV/EHU),
\par
Barrio Sarriena S/N 48940 Leioa, Spain.
\par
\smallskip
e-mail: \texttt{andrea.merlo@ehu.eus}
\par
}

\end{document}